\documentclass[12pt]{amsart}

\usepackage{amssymb, amsmath, amsthm}
\usepackage{tikz, graphics, graphicx}

\allowdisplaybreaks

\numberwithin{equation}{section}

\theoremstyle{plain}
\newtheorem{thm}{Theorem}[section]
\newtheorem{prop}[thm]{Proposition}

\newtheorem{cor}[thm]{Corollary}

\theoremstyle{definition}
\newtheorem{defn}[thm]{Definition}
\newtheorem{rem}[thm]{Remark}
\newtheorem{example}[thm]{Example}

\newtheorem*{acknowledgement}{Acknowledgement}

\newcommand{\ichi}{\mathbf{1}}
\newcommand{\C}{\mathbb{C}}
\newcommand{\N}{\mathbb{N}}
\newcommand{\R}{\mathbb{R}}

\newcommand{\Z}{\mathbb{Z}}
\newcommand{\calA}{\mathcal{A}}
\newcommand{\calB}{\mathcal{B}}

\newcommand{\calF}{\mathcal{F}}

\newcommand{\calM}{\mathcal{M}}
\newcommand{\calS}{\mathcal{S}}
\newcommand{\supp}{\mathrm{supp}\, }

\newcommand{\K}{\mathbb{K}}

\newcommand{\II}{I\hspace{-3pt}I}

\begin{document}
\title[Bilinear pseudo-differential operators
of $S_{0,0}$-type]
{
Boundedness  
of bilinear pseudo-differential operators 
of $S_{0,0}$-type on $L^2 \times L^2$ 
}

\author[T. Kato]{Tomoya Kato}
\author[A. Miyachi]{Akihiko Miyachi}
\author[N. Tomita]{Naohito Tomita}

\address[T. Kato and N. Tomita]
{Department of Mathematics, 
Graduate School of Science, Osaka University, 
Toyonaka, Osaka 560-0043, Japan}
\address[A. Miyachi]
{Department of Mathematics, 
Tokyo Woman's Christian University, 
Zempukuji, Suginami-ku, Tokyo 167-8585, Japan}

\email[T. Kato]{t.katou@cr.math.sci.osaka-u.ac.jp}
\email[A. Miyachi]{miyachi@lab.twcu.ac.jp}
\email[N. Tomita]{tomita@math.sci.osaka-u.ac.jp}

\date{\today}

\thanks{This work was supported by JSPS KAKENHI Grant Numbers 
JP17J00359 (Kato), JP16H03943 (Miyachi), and JP16K05201 (Tomita).}


\begin{abstract}
We extend the known result 
that the bilinear pseudo-differential 
operators with symbols in the 
bilinear H\"ormander class 
$BS^{-n/2}_{0,0}(\R^n)$ are bounded from 
$L^2 \times L^2$ to $h^1$. 
We show that 
those operators are also bounded from 
$L^2 \times L^2$ to 
$L^r $ for every $1< r\le 2$.  
Moreover we give similar results 
for symbol classes wider 
than $BS^{-n/2}_{0,0}(\R^n)$.  
We also give results  
for symbols of limited smoothness.  
\end{abstract}

\maketitle

\section{Introduction}\label{Introduction}

For a bounded measurable 
function $\sigma = \sigma (x, \xi_1, \xi_2)$ on $(\R^n)^3$,
the bilinear pseudo-differential operator
$T_{\sigma}$ is defined by
\[
T_{\sigma}(f_1,f_2)(x)
=\frac{1}{(2\pi)^{2n}}
\int_{(\R^n)^2}e^{i x \cdot(\xi_1+\xi_2)}
\sigma(x,\xi_1,\xi_2)\widehat{f_1}(\xi_1)
\widehat{f_2}(\xi_2)\, d\xi_1 d\xi_2
\]
for $f_1,f_2 \in \calS(\R^n)$.

For the boundedness of the bilinear operators 
$T_{\sigma}$, we shall 
use the following terminology. 
Let $X_1,X_2$, and $Y$ be function spaces on $\R^n$ 
equipped with quasi-norms 
$\|\cdot \|_{X_1}$, $\|\cdot \|_{X_2}$, 
and $\|\cdot \|_{Y}$, 
respectively.  
If there exists a constant $A$ such that 
\begin{equation}\label{boundedness-X_1X_2Y}
\|T_{\sigma}(f_1,f_2)\|_{Y}
\le A \|f_1\|_{X_1} \|f_2\|_{X_2} 
\;\;
\text{for all}
\;\;
f_1\in \calS \cap X_1  
\;\;
\text{and}
\;\;
f_2\in \calS \cap X_2,  
\end{equation}
then, 
with a slight abuse of terminology, 
we say that 
$T_{\sigma}$ is bounded from 
$X_1 \times X_2$ to $Y$ 
and write 
$T_{\sigma}: X_1 \times X_2 \to Y$.  
The smallest constant $A$ of 
\eqref{boundedness-X_1X_2Y} 
is denoted by 
$\|T_{\sigma}\|_{X_1 \times X_2 \to Y}$. 
If $\calA$ is a class of symbols,  
we denote by $\mathrm{Op}(\calA)$
the class of all bilineaer 
operators $T_{\sigma}$ 
corresponding to $\sigma \in \calA$. 
If $T_{\sigma}: X_1 \times X_2 \to Y$ 
for all $\sigma \in \calA$, 
then we write 
$
\mathrm{Op}(\calA) 
\subset B (X_1 \times X_2 \to Y)$.

The bilinear H\"ormander symbol class
$BS^m_{\rho,\delta}=
BS^m_{\rho,\delta}(\R^n)$,
$m \in \R$, $0 \le \rho, \delta \le 1$,
consists of all
$\sigma(x,\xi_1,\xi_2) \in C^{\infty}((\R^n)^3)$
such that
\[
|\partial^{\alpha}_x\partial^{\beta_1}_{\xi_1}
\partial^{\beta_2}_{\xi_2}\sigma(x,\xi_1,\xi_2)|
\le C_{\alpha,\beta_1,\beta_2}
(1+|\xi_1|+|\xi_2|)^{m+\delta|\alpha|-\rho(|\beta_1|+|\beta_2|)}
\]
for all multi-indices
$\alpha,\beta_1,\beta_2 \in \N_0^n
=\{0, 1, 2, \dots \}^n$.

In the case $\rho=1$ and $\delta <1$, 
the bilinear pseudo-differential operators with symbols 
in $BS^0_{1,\delta}$ are bilinear Calder\'on--Zygmund operators 
in the sense of Grafakos-Torres \cite{GT} 
and they are bounded from $L^p \times L^q$ to 
$L^r$ with $1<p, q<\infty$ and $1/r = 1/p+ 1/q$ 
(see Coifman-Meyer \cite{CM}, 
B\'enyi-Torres \cite{BT-2003}, and 
B\'enyi-Maldonado-Naibo-Torres \cite{BMNT}). 
Here the condition $1/r=1/p+1/q$ is necessary 
since the constant function belongs to 
$BS^0_{1,\delta}$ and the operator $T_{\sigma}$ 
corresponding to $\sigma=1$ is simply the pointwise product 
of functions.


In this paper, we shall be interested in the 
case $\rho= \delta=0$ and consider only 
the boundedness of $T_{\sigma}$ on $L^2 \times L^2$.  
Recall that $BS^{m}_{0,0} (\R^n)$ 
consists of all $\sigma$ satisfying the estimate 
\begin{equation}\label{defBSm00}
|\partial^{\alpha}_x\partial^{\beta_1}_{\xi_1}
\partial^{\beta_2}_{\xi_2}\sigma(x,\xi_1,\xi_2)|
\le C_{\alpha,\beta_1,\beta_2}
(1+|\xi_1|+|\xi_2|)^{m}.  
\end{equation}

Bilinear pseudo-differential operators with symbols in 
$BS^{m}_{0,0} (\R^n)$ have some 
features different from the corresponding   
linear operators. 
For the case of linear pseudo-differential operator, 
which is defined by 
\[
\sigma(X,D)f(x)
=\frac{1}{(2\pi)^n}
\int_{\R^n}e^{ix\cdot\xi}
\sigma(x,\xi)\widehat{f}(\xi)\, d\xi,
\quad f \in \calS(\R^n),
\]
the celebrated Calder\'on-Vaillancourt theorem 
states that the operator 
$\sigma (X,D)$ is bounded on $L^2(\R^n)$ 
if the symbol $\sigma(x,\xi)$ satisfies the estimate 
\[
|\partial^{\alpha}_x\partial^{\beta}_{\xi}\sigma(x,\xi)|
\le C_{\alpha,\beta}
\]
for all multi-indices $\alpha,\beta \in \N_0^n$  
(see \cite {CV}). 
For bilinear operators, 
innocent generalization of this theorem does not hold. 
In fact, B\'enyi-Torres \cite{BT-2004}
proved that there exists a symbol in $BS^0_{0,0}$ 
for which the corresponding bilinear 
pseudo-differential operator is not bounded 
from $L^2 \times L^2$ to $L^1$. 
Thus in order to have the inclusion 
$\mathrm{Op} (BS^{m}_{0,0}) \subset 
B (L^2 \times L^2 \to L^1)$, 
the order $m$ must be negative.
Miyachi--Tomita \cite{MT-2013} proved 
that the inclusion 
$\mathrm{Op} (BS^{m}_{0,0} ( \R^n )) \subset 
B (L^2  \times L^2 \to L^1)$ 
holds 
if and only if $m \le -n/2$.
For the critical case $m=-n/2$, 
it is also proved in \cite{MT-2013} 
that 
\begin{equation}\label{BS00critical}
\mathrm{Op} (BS^{-n/2}_{0,0} ( \R^n )) \subset 
B (L^2 \times L^2 \to h^1), 
\end{equation}
where $h^1$ is the 
local Hardy space of Goldberg \cite{goldberg 1979} 
(the definition of $h^1$ will be given in the next section).

The purpose of the present paper is to improve 
\eqref{BS00critical} in three ways. 
Firstly, we show that the target space 
$h^1$ in  
\eqref{BS00critical} 
can be replaced by $L^r$ 
with $1<r\le 2$ or even by 
the amalgam space $(L^2, \ell^1)$.  
(The definition of the amalgam space is given in the 
next section.)     
Since $(L^2, \ell^1) \hookrightarrow h^1 \cap L^2$, 
this is an improvement of 
\eqref{BS00critical}.  
Secondly, we show that 
the class $BS^{-n/2}_{0,0} ( \R^n ) $ 
can be replaced by a general class. 
We show that the weight 
function $(1+|\xi_1|+|\xi_2|)^{-n/2}$ 
appearing in the definition of 
$BS^{-n/2}_{0,0} ( \R^n ) $ (see \eqref{defBSm00})  
can be replaced by other functions 
and, among functions that have certain 
moderate behavior, 
we shall characterize all the possible weight functions. 
Thirdly, we give some refined results   
concerning operators with symbols of limited smoothness.

To explain our results in more detail, 
we introduce the following.

\begin{defn}\label{BSW00} 
For a nonnegative bounded function 
$W$ on $\R^n \times \R^n$, 
we denote by 
$BS^{W}_{0,0}(\R^n)$ the set of all those 
smooth functions 
$\sigma = \sigma (x, \xi_1, \xi_2)$ 
on $\R^n \times \R^n \times \R^n$ 
such that the estimate 
\[
|\partial^{\alpha}_x 
\partial^{\beta_1}_{\xi_1} 
\partial^{\beta_2}_{\xi_2} 
\sigma(x,\xi_1,\xi_2)|
\le C_{\alpha,\beta_1,\beta_2}
W (\xi_1, \xi_2)
\]
holds for all multi-indices
$\alpha, \beta_1, \beta_2 \in \N_0^n$. 
We shall call $W$ the {\it weight function\/} 
of the class $BS^{W}_{0,0}(\R^n)$. 
\end{defn}

\begin{defn}\label{classB}
We denote by 
$\calB (\Z^n \times \Z^n)$ 
the set of 
all those nonnegative functions 
$V$ on $\Z^n \times \Z^n$  
for which there exists a constant $c\in (0, \infty)$ 
such that the inequality 
\begin{equation}\label{BL222}
\sum_{\nu_1, \nu_2 \in \Z^n} 
V(\nu_1, \nu_2) 
A(\nu_1+ \nu_2) 
B(\nu_1)
C(\nu_2)
\le c 
\|A\|_{\ell^2 (\Z^n) } 
\|B\|_{\ell^2 (\Z^n) } 
\|C\|_{\ell^2 (\Z^n) } 
\end{equation}
holds for all nonnegative 
functions $A, B, C$ on $\Z^n$. 
\end{defn}

Now the following is one of the main theorems 
of this paper.

\begin{thm}\label{main-thm-1}
Let $V$ be a nonnegative bounded function 
on $\Z^n \times \Z^n$ 
and let 
\[
\widetilde{V} (\xi_1, \xi_2)
=
\sum_{\nu_1, \nu_2 \in \Z^n} 
V(\nu_1, \nu_2) 
\ichi_{Q} (\xi_1 - \nu_1) 
\ichi_{Q} (\xi_2 - \nu_2), 
\quad 
(\xi_1, \xi_2) \in \R^n \times \R^n,  
\]
where $Q=[-1/2, 1/2)^n$. 
Then the following hold. 
\\
$(1)$ 
If there exists an $r \in (0, \infty)$ such that 
all $T_{\sigma} \in \mathrm{Op}
(BS^{\widetilde{V}}_{0,0} (\R^n))$ 
are 
bounded from 
$L^2 \times L^2 $ 
to $L^r $, 
then $V \in \calB (\Z^n \times \Z^n)$. 
\\
$(2)$ 
Conversely,  
if $V \in \calB (\Z^n \times \Z^n)$, 
then all 
$T_{\sigma} \in \mathrm{Op}
(BS^{\widetilde{V}}_{0,0} (\R^n))$ 
are 
bounded from 
$L^2  \times L^2 $ 
to the amalgam space 
$(L^2, \ell^1)$. 
In particular, all those 
$T_{\sigma}$ are  
bounded from 
$L^2 \times L^2 $ 
to $L^r $ for all $r\in [1,2]$ and 
to $h^1 $. 
\end{thm}

Some typical examples of functions in 
$\calB (\Z^n \times \Z^n)$ are 
the following.

\begin{example}\label{example-of-V}
The following functions $V$ on 
$\Z^n \times \Z^n$ belong to 
the class $\calB (\Z^n \times \Z^n)$: 
\begin{align}
&\label{V-n/2}
V(\nu_1, \nu_2)= (1 + |\nu_1|+ |\nu_2|)^{-n/2};   
\\
&\label{Vproduct2}
V(\nu_1, \nu_2) = 
(1+ |\nu_{1}| )^{-a_{1}} 
(1+ |\nu_{2}| )^{-a_{2}}, 
\quad 
a_1, a_2>0, \quad a_1+a_2=n/2;   
\\
&\label{Vproduct1}
V(\nu_1, \nu_2) = 
\prod_{j=1}^{n}
\prod_{i=1}^{2} 
(1+ | \nu_{i,j} | )^{-a_{i,j}}  
\quad 
a_{i,j}>0, \quad 
a_{1,j }+a_{2, j}=1/2;  
\end{align}
where 
$\nu_i = (\nu_{i,1}, \dots, \nu_{i,n}) \in \Z^n$, 
$i=1, 2$. 
\end{example}

Notice that 
the bilinear H\"ormander class 
$BS^{-n/2}_{0,0}(\R^n)$ is equal to 
the class $BS^{ \widetilde{V} }_{0,0}(\R^n)$ of 
Theorem \ref{main-thm-1} 
with $V$ of \eqref{V-n/2}. 
Observe that the function 
\eqref{Vproduct2} is bigger than  
\eqref{V-n/2} and 
\eqref{Vproduct1} is much bigger, 
and hence the corresponding classes  
 $BS^{ \widetilde{V} }_{0,0}(\R^n)$ are wider than 
$BS^{-n/2}_{0,0}(\R^n)$. 
We shall prove that not only 
\eqref{V-n/2} but also 
any $V$ in the Lorentz class $\ell^{4,\infty} (\Z^{2n})$ 
belongs to $\calB (\Z^n \times \Z^n)$. 
We also prove $\calB (\Z^n \times \Z^n)$ contains 
functions that are generalizations of  
\eqref{Vproduct2} and \eqref{Vproduct1}.

It will be worthwhile to observe that 
the claim of Theorem \ref{main-thm-1} (2) 
for $V$ of \eqref{Vproduct2} 
is equivalent to the following: 
the bilinear pseudo-differential 
operators $T_{\sigma}$ with 
$\sigma \in BS^0_{0,0}(\R^n)$ are 
bounded from $W^{a_1} \times W^{a_2}$ to 
$(L^2 ,\ell^1 )
\hookrightarrow h^1  \cap L^2 $ 
for all $a_1, a_2$ satisfying the conditions of  \eqref{Vproduct2}, 
where $W^s=W^s (\R^n)$ denotes the 
$L^2$-based Sobolev space.

Recently 
Grafakos--He--Slav\'ikov\'a \cite{GHS} proved 
that if the symbol 
$\sigma (x, \xi, \eta)=\sigma (\xi, \eta)$ does 
not depend on $x$, and if 
$\sigma \in BS^{0}_{0,0}(\R^n) \cap L^{q}(\R^{2n})$ 
with $q<4$, 
then 
$T_{\sigma}$ is bounded from 
$L^2 \times L^2 $ to $L^1 $. 
In the present paper, 
we shall show that this result, 
even in a generalized form, 
can be deduced from Theorem \ref{main-thm-1}.

Not only Theorem \ref{main-thm-1}, 
we also give refined theorems   
which treat symbols of limited smoothness. 
For linear pseudo-differential operators, 
there are several results concerning 
symbols with limited smoothness. 
Authors such
as 
Cordes \cite{Cordes}, 
Coifman-Meyer \cite{CM}, 
Muramatu \cite{Muramatu}, 
Miyachi \cite{Miyachi}, 
Sugimoto \cite{Sugimoto}, 
and Boulkhemair \cite{boulkhemair 1995}  
investigated minimal smoothness assumptions  
on the symbols to assure 
the boundedness of linear pseudo-differential operators. 
As for the $L^2$ boundedness, 
they proved that, roughly speaking, 
smoothness of symbols up to 
$n/2$ for each variable $x$ and $\xi$ 
assures the boundedness in $L^2$. 
For bilinear operators, to the best of the authors' knowledge, 
there is only one result 
concerning symbols of limited smoothness, 
which was given by Herbert--Naibo \cite{herbert naibo 2016}. 
In \cite{herbert naibo 2016}, the authors proved 
that symbols of the class $BS^m_{0,0}(\R^n)$ with $m < -n/2$  
provide bounded bilinear pseudo-differential 
operators in $L^2 \times L^2 \to L^1$ 
if the smoothness up to $n/2$ for the $x$ variable 
and up to $n$ for the $\xi_1$ and $\xi_2$ variables are assumed.
In the present paper, we 
shall relax the smoothness condition of 
\cite{herbert naibo 2016} 
and also give results for general classes which include 
$BS^m_{0,0}(\R^n)$ of critical order $m=-n/2$.

Our method to prove the boundedness 
of pseudo-differential operators 
relies on the idea of 
Boulkhemair \cite{boulkhemair 1995}, 
who treated linear pseudo-differential operators.

We end this section by mentioning the plan of this paper.
In Section \ref{section2}, 
we will give the basic notations used throughout this paper 
and recall the definitions and properties of some function spaces.
In Section \ref{sectionWeight}, 
we give several properties of the class 
$\calB (\Z^n \times \Z^n)$ and prove 
that it contains 
the functions $V$ of Example \ref{example-of-V}.  
In Section \ref{sectionMain}, 
we prove Theorem \ref{main-thm-1} and 
also give two other main theorems of this paper, 
Theorems \ref{main-thm} and 
\ref{main-thm-2}. 
The latter theorems treat symbols with 
limited smoothness. 
In the same section, 
we also give a proof to the theorem  
of Grafakos--He--Slav\'ikov\'a \cite{GHS} 
by using Theorem \ref{main-thm-1}.  
In Section \ref{sectionsharpness},
we show the sharpness of our main theorems.

\section{Preliminaries}\label{section2}

\subsection{Basic notations} 

We collect notations which will be used throughout this paper.
We denote by $\R$, $\Z$, $\N$, and $\N_0$
the sets of real numbers, integers, positive integers, 
and nonnegative integers, respectively. 
We denote by $Q$ the $n$-dimensional 
unit cube $[-1/2,1/2)^n$. 
For $1 \leq p \leq \infty$, $p^\prime$ is 
the conjugate number of $p$ defined by $1/p + 1/p^\prime =1$.
We write 
$[s] = \max\{ n \in \Z : n \leq s \}$ for $s \in \R$. 
For $x\in \R^d$, we write 
$\langle x \rangle = (1 + | x |^2)^{1/2}$. 
Thus $\langle (x,y) \rangle = (1+|x|^2+|y|^2)^{1/2}$ 
for $(x,y) \in \R^n \times \R^n$.

For two nonnegative functions $A(x)$ and $B(x)$ defined 
on a set $X$, 
we write $A(x) \lesssim B(x)$ for $x\in X$ to mean that 
there exists a positive constant $C$ such that 
$A(x) \le CB(x)$ for all $x\in X$. 
We often omit to mention the set $X$ when it is 
obviously recognized.  
Also $A(x) \approx B(x)$ means that
$A(x) \lesssim B(x)$ and $B(x) \lesssim A(x)$.

We denote the Schwartz space of rapidly 
decreasing smooth functions
on $\R^d$ 
by $\calS (\R^d)$ 
and its dual,
the space of tempered distributions, 
by $\calS^\prime(\R^d)$. 
The Fourier transform and the inverse 
Fourier transform of $f \in \calS(\R^d)$ are given by
\begin{align*}
\mathcal{F} f  (\xi) 
&= \widehat {f} (\xi) 
= \int_{\R^d}  e^{-i \xi \cdot x } f(x) \, dx, 
\\
\mathcal{F}^{-1} f (x) 
&= \check f (x)
= \frac{1}{(2\pi)^d} \int_{\R^d}  e^{i x \cdot \xi } f( \xi ) \, d\xi,
\end{align*}
respectively. 
For $m \in \calS^\prime (\R^d)$, 
the Fourier multiplier operator is defined by
\begin{equation*}
m(D) f 
=
\mathcal{F}^{-1} \left[m \cdot \mathcal{F} f \right].
\end{equation*}
We also use the notation $(m(D)f)(x)=m(D_x)f(x)$ 
when we indicate which variable is considered.

For a measurable subset $E \subset \R^d$, 
the Lebesgue space $L^p (E)$, $0<p\le \infty$, 
is the set of all those 
measurable functions $f$ on $E$ such that 
$\| f \|_{L^p(E)} = 
\left( \int_{E} \big| f(x) \big|^p \, dx \right)^{1/p} 
< \infty 
$
if $0< p < \infty$ 
or   
$\| f \|_{L^\infty (E)} 
= 
\mathrm{ess}\, \sup_{x\in E} |f(x)|
< \infty$ if $p = \infty$. 
We also use the notation 
$\| f \|_{L^p(E)} = \| f(x) \|_{L^p_{x}(E)} $ 
when we want to indicate the variable explicitly.

The uniformly local $L^2$ space, denoted by 
$L^2_{ul} (\R^n)$, consists of 
all those measurable functions $f$ on 
$\R^n$ such that 
\begin{equation*}
\| f \|_{L^2_{ul} (\R^n) } 
= \sup_{\nu \in \Z^n}
\left( \int_{Q} 
\big| f(x+\nu) 
\big|^2 \, dx 
\right)^{1/2}
< \infty 
\end{equation*}
(this notion can be found in 
\cite[Definition 2.3]{kato 1975}).

%

Let $\K$ be a countable set. 
We define the sequence spaces 
$\ell^q (\K)$ and $\ell^{q, \infty} (\K)$ 
as follows.  
The space $\ell^q (\K)$, $1 \le q \le \infty$, 
consists of all those 
complex sequences $a=\{a_k\}_{k\in \K}$ 
such that 
$ \| a \|_{ \ell^q (\K)} 
= 
\left( \sum_{ k \in \K } 
| a_k |^q \right)^{ 1/q } <\infty$ 
if $1 \leq q < \infty$  
or 
$\| a \|_{ \ell^\infty (\K)} 
= \sup_{k \in \K} |a_k| < \infty$ 
if $q = \infty$.
For $1\le q<\infty$, 
the space 
$\ell^{ q,\infty }(\K)$ is 
the set of all those complex 
sequences 
$a= \{ a_k \}_{ k \in \K }$ such that
\begin{equation*}
\| a \|_{\ell^{q,\infty} (\K) } 
=\sup_{t>0} 
\big\{ t\, 
\sharp \big( \{ k \in \K : | a_k | > t \} 
\big)^{1/q} \big\} < \infty,
\end{equation*}
where $\sharp$ denotes the cardinality 
of a set. 
Sometimes we write 
$\| a \|_{\ell^{q}} = 
\| a_k \|_{\ell^{q}_k }$ or 
$\| a \|_{\ell^{q,\infty}} = 
\| a_k \|_{\ell^{q,\infty}_k }$.  
If $\K=\Z^n$, 
we usually write $ \ell^q $ or $\ell^{ q, \infty }$ 
for 
$\ell^q (\Z^n)$  
or $\ell^{q, \infty} (\Z^n)$.

Let $X,Y,Z$ be function spaces.  
We denote the mixed norm by
\begin{equation*}
\label{normXYZ}
\| f (x,y,z) \|_{ X_x Y_y Z_z } 
= 
\bigg\| \big\| \| f (x,y,z) 
\|_{ X_x } \big\|_{ Y_y } \bigg\|_{ Z_z }.
\end{equation*} 
(Here pay special attention to the order 
of taking norms.)  
We shall use these mixed norms for  
$X, Y, Z$ being $L^p$ or $\ell^p$. 
Recall that the Minkowski inequality implies 
\begin{equation}\label{mixedMinkowski}
\|f(x,y)\|_{L^p_x L^q_y}\le 
\|f(x,y)\|_{L^q_y L^p_x}, 
\;\; \text{if} \; \; p\le q.
\end{equation}

\subsection{Local Hardy space $h^1$ and the space $bmo$} 
\label{sechardybmo}

We recall the definition of the local Hardy space 
$h^1(\R^n)$ and the space $bmo(\R^n)$.

Let $\phi \in \calS(\R^n)$ be such that
$\int_{\R^n}\phi(x)\, dx \neq 0$. 
Then, the local Hardy space $h^1(\R^n)$ 
consists of
all $f \in \calS'(\R^n)$ such that 
$\|f\|_{h^1}=\|\sup_{0<t<1}|\phi_t*f|\|_{L^1}
<\infty$,  
where $\phi_t(x)=t^{-n}\phi(x/t)$.
It is known that $h^1(\R^n)$
does not depend on the choice of the function $\phi$,
and that $h^1(\R^n) \hookrightarrow L^1(\R^n)$.

The space $bmo(\R^n)$ consists of
all locally integrable functions $f$ on $\R^n$ 
such that
\[
\|f\|_{bmo}
=\sup_{|R| \le 1}\frac{1}{|R|}
\int_{R}|f(x)-f_R|\, dx
+\sup_{|R|\geq1}\frac{1}{|R|}
\int_R |f(x)|\, dx
<\infty,
\]
where $f_R=|R|^{-1}\int_R f$,
and $R$ ranges over the cubes in $\R^n$.

It is known that the dual space of $h^1(\R^n)$ is $bmo(\R^n)$.
See Goldberg \cite{goldberg 1979} for more details about 
$h^1$ and $bmo$. 

\subsection{Amalgam spaces} 
\label{secamalgam}

For $1 \leq p,q \leq \infty$, 
the amalgam space 
$(L^p,\ell^q)(\R^n)$ is 
defined to be the set of all 
those measurable functions $f$ on 
$\R^n$ such that 
\begin{equation*}
\| f \|_{ (L^p,\ell^q) (\R^n)} 
=
 \| f(x+\nu) \|_{L^p_x (Q) \ell^q_{\nu}(\Z^n) }
=\left\{ \sum_{\nu \in \Z^n}
\left( \int_{Q} \big| f(x+\nu) \big|^p \, dx 
\right)^{q/p} \right\}^{1/q} 
< \infty  
\end{equation*}
with usual modification when $p$ or $q$ is infinity.  
Obviously, $(L^p,\ell^p) = L^p$ 
and $(L^2, \ell^\infty) = L^2_{ul}$. 
For $1\leq p,q < \infty$, 
the duality $(L^p,\ell^q)^\ast =  (L^{p'},\ell^{q'}) $ holds.
If $p_1 \geq p_2$ and $q_1 \leq q_2$, 
then 
$(L^{p_1}, \ell^{q_1}) 
\hookrightarrow (L^{p_2},\ell^{q_2}) $.
In particular, 
$(L^2,\ell^r) \hookrightarrow L^r$ 
for $1 \leq r \leq 2$. 
In the case $r=1$, the stronger embedding 
$(L^2,\ell^1) \hookrightarrow h^1$ holds. 
This last fact follows from 
the embedding $bmo \hookrightarrow (L^2,\ell^\infty)$ 
and the duality $(h^1)^{\prime}=bmo$. 
For 
$1\le p, q_1, \dots, q_n \le \infty$, 
we also define the space 
$(L^p,\ell^{q_1} \dots \ell^{q_n} ) (\R^n)$  
by the mixed norm 
\begin{equation*}
\| f \|_{ (L^p,\, \ell^{q_1} \dots \ell^{q_n} ) ( \R^n )} 
=
 \| f(x+\nu) \|_{L^p_x (Q) 
\ell^{q_1}_{\nu_1}(\Z) 
\dots 
\ell^{q_n}_{\nu_n}(\Z) 
}, 
\end{equation*}
where 
$\nu = (\nu_1, \dots, \nu_n) \in \Z^n$. 
See Fournier--Stewart \cite{fournier stewart 1985} 
and Holland \cite{holland 1975} for more 
properties of amalgam spaces.

\section{Class $\calB$}
\label{sectionWeight}

In this section, we give several properties of the class 
$\calB (\Z^n \times \Z^n)$ introduced in 
Definition \ref{classB}. 
We also introduce the class $\calM (\R^d)$, 
which will be used in the next section.

\begin{prop}\label{propertiesB}
\begin{enumerate}
\setlength{\itemindent}{0pt} 
\setlength{\itemsep}{3pt} 
\item 
Every function in the class 
$\calB (\Z^n \times \Z^n )$ 
is bounded. 
\item 
A nonnegative function $V=V(\nu_1, \nu_2)$ 
on $\Z^n \times \Z^n$ belongs to 
$\calB (\Z^n \times \Z^n )$ if and only 
if 
$V(\nu_1+ \nu_2, -\nu_2)$ 
or $V(- \nu_1, \nu_1+ \nu_2)$ 
belongs to $\calB (\Z^n \times \Z^n )$.

\item 
The class $\calB (\Z^n \times \Z^n)$ is not 
rearrangement invariant, 
i.e., 
there exists a function $V$ on $\Z^n \times \Z^n$ 
and a bijection $\Phi : \Z^n \times \Z^n \to \Z^n \times \Z^n$ 
such that  
$V \in \calB (\Z^n \times \Z^n)$ but 
$V\circ \Phi \not\in \calB (\Z^n \times \Z^n)$. 

\item 
 Let $d, d' \in \N$, 
$V\in \calB (\Z^{d}\times \Z^{d})$, 
and 
$V'\in \calB (\Z^{d'}
\times \Z^{d'})$.  
Then 
the function 
\[
W((\mu_1, \mu'_1), (\mu_2, \mu'_2))  
= V (\mu_1, \mu_2) 
V^{\prime} (\mu'_1, \mu'_2), 
\quad 
\mu_1, \mu_2 \in \Z^{d}, 
\quad 
\mu'_1, \mu'_2 
\in \Z^{d^{\prime}}, 
\]
belongs to 
$\calB (\Z^{ d+d^{\prime} }\times \Z^{ d+d^{\prime} })$. 
\end{enumerate}
\end{prop}

\begin{proof}
(1) If $V$ satisfies \eqref{BL222}, 
then applying it to the case where each of 
$A, B, C$ is a defining function  
of one point we easily find $V (\nu_1, \nu_2)\le c$.

(2) This can be easily proved by a simple change of 
variables.

(3) 
First observe that 
the function 
$V (\nu_1, \nu_2) = \langle \nu_1 
\rangle ^{-n/2- \epsilon}$ with $\epsilon >0$ 
belongs to $\calB (\Z^n \times \Z^n)$. 
In fact 
for this $V$ and for $B(\nu_1)\in \ell^2_{\nu_1} (\Z^n)$, 
the function $VB$ belongs to 
$\ell^{1} (\Z^n)$ and the 
inequality \eqref{BL222} can be easily checked 
by the use of H\"older's inequality. 
(See also Proposition \ref{single-L2} below.)     
On the other hand, for 
$\alpha > 0$, 
the function 
\[
W (\nu_1, \nu_2) = 
\langle (\nu_1, \nu_2) \rangle ^{-n/2 + \alpha}, 
\quad (\nu_1, \nu_2) \in \Z^n \times \Z^n  
\]
does not belong to $\calB (\Z^n \times \Z^n)$. 
In fact, for $A(\mu)=B(\mu)=C(\mu)= 
\langle \mu \rangle^{-n/2-\alpha/4}
\in \ell^{2}_{\mu}(\Z^n)$, it is easy to see that 
$\sum W(\nu_1, \nu_2) A(\nu_1 + \nu_2) B(\nu_1) C(\nu_2)
=\infty$. 
For $j\in \N_0$, set 
\begin{align*}
&
E_j (V)= \{
(\nu_1, \nu_2)  \in \Z^n \times \Z^n 
\mid 
2^{-j-1}< V(\nu_1, \nu_2)  \le 2^{-j}
\}, 
\\
&
E_j (W)= \{
(\nu_1, \nu_2)  \in \Z^n \times \Z^n 
\mid 
2^{-j-1}< W(\nu_1, \nu_2)  \le 2^{-j}
\}. 
\end{align*}
Then both 
$\{E_j(V)\}_{j \in \N_0}$ and $\{E_j(W)\}_{j \in \N_0}$ 
are partitions of $\Z^n \times \Z^n$, 
each 
$E_j(V)$ is an infinite set, 
and 
$E_j(W)$ is a finite set.   
It is easy to construct a 
bijection $\Phi$ of $\Z^n \times \Z^n$ onto itself 
such that 
\begin{equation*}
\Phi (E_{j}(W)) \subset E_{0}(V) \cup \dots \cup E_{j}(V) \
\;\; \text{for all}\;\; 
j \in \N_0. 
\end{equation*}
Then 
$W \le 2^{-j}$ 
and 
$V\circ \Phi > 2^{-j-1}$  
on each $E_{j}(W)$, we have 
$W < 2 V\circ \Phi $ on the whole $\Z^n \times \Z^n$. 
Since $W \not\in \calB (\Z^n \times \Z^n)$, we have  
$V\circ \Phi  \not\in \calB (\Z^n \times \Z^n)$.

(4) 
Let $A, B, C$ be nonnegative functions 
on $\Z^{d+d'} $ and 
consider the sum 
\begin{equation*}
\sum_{
(\mu_1, \mu'_1),  (\mu_2, \mu'_2) \in \Z^{d+d'}
} \, 
V(\mu_1, \mu_2)
V'(\mu'_1, \mu'_2)
A((\mu_1, \mu'_1)+ (\mu_2, \mu'_2)) 
B(\mu_1, \mu'_1)
C(\mu_2, \mu'_2). 
\end{equation*}
If we first take the sum over 
$\mu_1, \mu_2 \in \Z^d$, then 
the assumption $V\in \calB (\Z^d \times \Z^d)$ 
implies that the above sum is bounded by a constant times 
\[
\sum_{
\mu'_1,  \mu'_2 \in \Z^{d'} 
} \, 
V'(\mu'_1, \mu'_2)
\|A(\mu_1, \mu'_1+\mu'_2) \|_{\ell^2_{\mu_1}(\Z^{d})}
\|B(\mu_1, \mu'_1)\|_{\ell^2_{\mu_1}(\Z^{d})}
\|C(\mu_2, \mu'_2)\|_{\ell^2_{\mu_2}(\Z^{d})}. 
\]
Now 
$V'\in \calB (\Z^{d'} \times \Z^{d'})$ 
implies that the last sum is bounded by a constant times 
\begin{align*}
&
\|A(\mu_1, \mu'_1) \|_{\ell^2_{\mu_1}(\Z^{d}) \ell^2_{\mu'_1} (\Z^{d'})}
\|B(\mu_1, \mu'_1)\|_{
\ell^2_{\mu_1} (\Z^{d})
\ell^2_{\mu'_1} (\Z^{d'})}
\|C(\mu_2, \mu'_2)\|_{
\ell^2_{\mu_2}(\Z^{d}) 
\ell^2_{\mu'_2}(\Z^{d'}) }
\\
&=
\|A\|_{\ell^2 (\Z^{d+d'})}
\|B\|_{\ell^2 (\Z^{d+d'})}
\|C\|_{\ell^2 (\Z^{d+d'})}. 
\end{align*}
Thus the function $W$ of (4) 
belongs to $\calB (\Z^{d+d'} \times \Z^{d+d'})$. 
\end{proof}

\begin{prop}\label{single-L2}
Suppose 
a nonnegative function $V$ on 
$\Z^n \times \Z^n$ is  
one of the following forms: 
\[
V(\nu_1, \nu_2)= 
V_{0}(\nu_1), \;
V_{0}(\nu_2), \;
V_{0}(\nu_1+ \nu_2). 
\] 
Then $V \in \calB (\Z^n \times \Z^n)$ 
if and only if $V_0 \in \ell^2 (\Z^n)$. 
In particular, 
a nonzero constant function does not belong to 
$\calB (\Z^n \times \Z^n )$. 
\end{prop}

\begin{proof}
We use the following fact: 
if $K$ is a nonnegative function on $\Z^n$, 
then the inequality 
\begin{equation}\label{nonnegativekernel} 
\bigg\|
\sum_{\nu_2 \in \Z^n} 
K(\nu_1 - \nu_2) X(\nu_2) 
\bigg\|_{\ell^2 _{\nu_1}(\Z^n)}
\le 
c 
\|X\|_{\ell^2 (\Z^n)}
\end{equation}
holds for all nonnegative functions $X$ on $\Z^n$  
if and only if 
$\|K\|_{\ell^1 (\Z^n)} \le c$. 
Here is a proof. 
Consider the case where 
$K(\nu)=0$ except for finitely many $\nu$'s. 
Then, 
by the $L^2$ theory 
of Fourier analysis for periodic functions, 
it is easy to see that the inequality 
\eqref{nonnegativekernel} holds for all nonnegative 
$X$ if and only if 
the function 
$k(x) = \sum_{\nu \in \Z^n} 
K(\nu) e^{2\pi i \nu \cdot x}$ 
satisfies 
$\|k\|_{L^{\infty}(Q)}\le c$. 
But since $K$ is nonnegative, we have 
$\|k\|_{L^{\infty}(Q)}=\|K\|_{\ell^1 (\Z^n)}$ 
and thus $\|K\|_{\ell^1 (\Z^n)} \le c$. 
The general case follows by a limiting argument.

Now suppose $V(\nu_1, \nu_2)=V_0(\nu_1)$ and 
$V \in \calB (\Z^n \times \Z^n)$. 
Then, by a change of variables, 
the inequality \eqref{BL222} is written as 
\[
\sum_{\nu_1, \nu_2 \in \Z^n} 
V_0 (\nu_1 - \nu_2) 
A(\nu_1)
B(\nu_1 - \nu_2)
C(\nu_2)
\le 
c 
\|A\|_{\ell^2}
\|B\|_{\ell^2}
\|C\|_{\ell^2}. 
\]
By the fact mentioned above, 
this inequality holds if and only if 
$\|V_0 B\|_{\ell^1} \le c \|B\|_{\ell^2}$, 
which is equivalent to $\|V_0\|_{\ell^2} \le c$.

The cases $V(\nu_1, \nu_2)=V_0(\nu_2)$ 
and 
$V(\nu_1, \nu_2)=V_0(\nu_1+ \nu_2)$ are 
proved in a similar way or by the use of 
Proposition \ref{propertiesB} (2). 
\end{proof}

\begin{prop}\label{productLweak}
Let $2<p_1, p_2 < \infty$, 
$1/p_1 + 1/p_2 = 1/2$, 
and let 
$f_1 \in \ell^{p_1, \infty}(\Z^d)$ and 
$f_2 \in \ell^{p_2, \infty}(\Z^d)$ 
be nonnegative sequences. 
Then the functions 
$f_1 (\nu_1) f_2 (\nu_2)$, 
$f_1 (\nu_1+ \nu_2) f_2 (\nu_2)$, 
and 
$f_1 (\nu_1) f_2 (\nu_1+ \nu_2)$ 
belong to $\calB (\Z^d\times \Z^d)$.  
\end{prop}

\begin{proof} 
By Proposition \ref{propertiesB} (2), it is sufficient 
to prove that 
$f_1 (\nu_1) f_2 (\nu_2)$ belongs to 
$\calB (\Z^d \times \Z^d)$. 
Let $A, B, C$ be nonnegative functions on $\Z^d$.

We set 
\[
E_{i} (j) 
= 
\{
\nu \in \Z^d 
\mid 
2^{-(j+1)/p_i} < f_i (\nu) \le 2^{-j/p_i}
\}, 
\quad 
i=1,2, 
\;\;
j\in \Z. 
\]
Our assumption 
$f_i \in \ell^{p_i, \infty}(\Z^d)$ implies the estimate 
\begin{equation}\label{sharpEij}
\sharp (E_{i}(j)) \lesssim 2^{j}. 
\end{equation}
Since $\{E_i (j)\}_{j \in \Z}$ gives 
a decomposition of the 
set $\{ \nu \in \Z^d : f_i (\nu)>0\}$,  
the sum 
on the left hand side of \eqref{BL222} 
for $V(\nu_1, \nu_2) = f_1(\nu_1)f_2 (\nu_2)$ 
is written as 
\begin{align*}
&\sum_{\nu_1, \nu_2 \in \Z^d} 
f_1(\nu_1) f_2(\nu_2) 
A(\nu_1+ \nu_2) 
B(\nu_1)
C(\nu_2)
\\
&\approx 
\sum_{j_1, j_2 \in \Z} 
\, 
\sum_{\nu_1 \in E_1 (j_1)}\, 
\sum_{\nu_2 \in E_2 (j_2)}\, 
2^{-j_1 / p_1}
2^{-j_2 / p_2}
A(\nu_1+\nu_2)
B(\nu_1)
C(\nu_2). 
\end{align*}

Fix $j_1, j_2$ and 
consider the sum over $\nu_1\in E_1 (j_1)$ 
and $\nu_2\in E_2 (j_2)$. 
If $j_1 \le j_2$, then we apply 
the Cauchy--Schwarz 
inequality first to the sum over $\nu_2$ and 
then to the sum over $\nu_1 $ to obtain 
\begin{align*}
&\sum_{\nu_1 \in E_1 (j_1)}\, 
\sum_{\nu_2 \in E_2 (j_2)}\, 
2^{-j_1 / p_1}
2^{-j_2 / p_2}
A(\nu_1+\nu_2)
B(\nu_1)
C(\nu_2)
\\
&
\le 
\sum_{\nu_1 \in E_1 (j_1)}\, 
2^{-j_1 / p_1}
2^{-j_2 / p_2}
\|A\|_{\ell^2 (\Z^d)}
B(\nu_1)
\|C\|_{\ell^2 (E_2 (j_2))}
\\
&
\le 
2^{-j_1 / p_1}
2^{-j_2 / p_2}
( \sharp  ( E_1 (j_1) ) )^{1/2}
\|A\|_{\ell^2 (\Z^d)}
\|B\|_{\ell^2 (E_1 (j_1))}
\|C\|_{\ell^2 (E_2 (j_2))}  
\\
&\lesssim 
2^{-(j_2-j_1) / p_2}
\|A\|_{\ell^2 (\Z^d)}
\|B\|_{\ell^2 (E_1 (j_1))}
\|C\|_{\ell^2 (E_2 (j_2))}, 
\end{align*}
where the last $\lesssim $ 
follows from 
the estimate 
$\sharp (E_1 (j_1)) \lesssim 
2^{j_1}$ (see \eqref{sharpEij}) 
and the equality $1/p_1 + 1/p_2 = 1/2$.  
Similarly, if $j_1 > j_2$, then we apply 
the Cauchy--Schwarz 
inequality first to the sum over $\nu_1$ and 
then to the sum over $\nu_2$ to obtain 
the same estimate as above but with 
the factor $2^{-(j_2-j_1) / p_2}$ 
replaced by $2^{-(j_1-j_2) / p_1}$.

Thus in either case we have 
\begin{align*}
&
\sum_{\nu_1 \in E_1 (j_1)}\,  
\sum_{\nu_2 \in E_2 (j_2)}\,  
2^{-j_1/p_1}
2^{-j_2/p_2}
A(\nu_1+ \nu_2) 
B(\nu_1)
C(\nu_2)
\\
&\lesssim 
(2^{-|j_1-j_2| / p_2}
+ 2^{-|j_1-j_2| / p_1})
\|A\|_{\ell^2 (\Z^d)}
\|B\|_{\ell^2 (E_1 (j_1))}
\|C\|_{\ell^2 (E_2 (j_2))}. 
\end{align*}
By the Schur lemma, 
the sum of the above over 
${j_1, j_2\in \Z}$ 
is bounded by 
\[
\|A\|_{\ell^2 (\Z^d)}
\|B\|_{\ell^2 (E_1 (j_1)) \ell^2_{j_1}}
\|C\|_{\ell^2 (E_2 (j_2)) \ell^2_{j_2}} 
\le 
\|A\|_{\ell^2 (\Z^d)}
\|B\|_{\ell^2 (\Z^d)}
\|C\|_{\ell^2 (\Z^d)}. 
\]
(For the Schur lemma, see, e.g., 
\cite[Appendix A]{grafakos 2014m}.)
\end{proof}

\begin{prop}\label{weightL4weak}
All nonnegative functions in the class 
$\ell^{4, \infty} (\Z^d \times \Z^d)$  
belong to $\calB (\Z^d\times \Z^d)$.  
\end{prop}

\begin{proof} 
By appropriately extending functions 
on $\Z^d$
and $\Z^d \times \Z^d$ to 
functions on $\R^d$
and $\R^d \times \R^d$, 
it is sufficient to prove the inequality 
\begin{equation}\label{BL4infty222}
\begin{split}
&\int_{\R^d \times \R^d} 
V(x_1, x_2)
A(x_1+x_2)
B(x_1)
C(x_2)\, dx_1 dx_2
\\
&\lesssim 
\|V\|_{
L^{4, \infty}(\R^d\times \R^d) }
\|A\|_{L^2 (\R^d)}
\|B\|_{L^2 (\R^d)}
\|C\|_{L^2 (\R^d)}  
\end{split}
\end{equation}
for nonnegative 
measurable functions $V, A, B, C$ 
on the corresponding 
Euclidean spaces.
We shall derive this inequality from 
the inequality 
\begin{equation}\label{BLq0q1q2q3}
\begin{split}
&\int_{\R^d \times \R^d} 
V(x_1, x_2)
A(x_1+x_2)
B(x_1)
C(x_2)\, dx_1 dx_2
\\
&\lesssim 
\|V\|_{
L^{q_0}(\R^d\times \R^d) }
\|A\|_{L^{q_1} (\R^d)}
\|B\|_{L^{q_2} (\R^d)}
\|C\|_{L^{q_3} (\R^d)} 
\end{split}
\end{equation}
by using real interpolation. 
It is known that 
\eqref{BLq0q1q2q3} holds if and only if 
the following two conditions are satisfied: 
\begin{align}
&\label{scaling}
\frac{2}{q_0}
+
\frac{1}{q_1}
+
\frac{1}{q_2}
+
\frac{1}{q_3}
=2, 
\\
&\label{Youngconvolution}
0\le 
\frac{1}{q_i} \le 1 - 
\frac{1}{q_0} \le 1, 
\quad i=1,2,3.    
\end{align}
For the reader's convenience, here we give a proof 
of the fact that  
\eqref{BLq0q1q2q3} holds under the assumptions 
\eqref{scaling} and 
\eqref{Youngconvolution}. 
It is sufficient to show 
\begin{equation}\label{HolderYoung}
\|A(x_1+x_2) B(x_1) C(x_2)\|_{L^{q_0'}_{x_1, x_2}} 
\lesssim 
\|A\|_{L^{q_1}}
\|B\|_{L^{q_2}}
\|C\|_{L^{q_3}}. 
\end{equation}
In the case $q_0'=\infty$, \eqref{Youngconvolution} implies 
$q_1=q_2=q_3=\infty$ and \eqref{HolderYoung} is obvious. 
We assume $q_0' < \infty$. 
Take $\alpha, \beta, \gamma, \delta \in [1, \infty]$ that 
satisfy 
$1/\delta + 1/\gamma=1$ and 
$1+ 1/\delta = 1/\alpha + 1/\beta$. 
Then writing $\widetilde{B}(\mu)= B (-\mu)$ 
and using H\"older's inequality and 
Young's inequality for convolution, we have 
\begin{align*}
&
\|A(x_1+x_2) B(x_1) C(x_2)\|_{L^{q_0'}_{x_1, x_2}} ^{q_0'}
=
\int 
(A^{q_0'}\ast  \widetilde{B}^{q_0'})(x_2)
C(x_2)^{q_0'}
\, dx_2 
\\
&
\le 
\|A^{q_0'}\ast  \widetilde{B}^{q_0'}\|_{ L^\delta }
\|C^{q_0'}\|_{ L^\gamma }
\le 
\|A^{q_0'}\|_{L^\alpha}
\|\widetilde{B}^{q_0'}\|_{L^\beta}
\|C^{q_0'}\|_{ L^\gamma }
\\
&=
\big(
\|A\|_{ L^{\alpha q_0'} }
\|B\|_{ L^{\beta q_0'} }
\|C\|_{ L^{\gamma q_0'} }
\big)^{q_0'}. 
\end{align*} 
By choosing $\alpha, \beta, \gamma$ such that 
$\alpha q_0'=q_1$, 
$\beta q_0'=q_2$, and  
$\gamma q_0'=q_3$, 
we obtain \eqref{HolderYoung} 
with the constant in $\lesssim$ equal to $1$.

From \eqref{BLq0q1q2q3}, it follows by duality 
that 
the trilinear map 
\[
T(A, B, C)(x_1, x_2)
=A (x_1+ x_2) B(x_1)C(x_2)
\]
satisfies the estimate 
\begin{equation*}
\|T(A,B,C)\|_{L^{q'_0} (\R^d\times \R^d) }
\lesssim 
\|A\|_{L^{q_1} (\R^d) }
\|B\|_{L^{q_2} (\R^d) }
\|C\|_{L^{q_3} (\R^d) }
\end{equation*}
for all $(q_i)$ 
satisfying \eqref{scaling} 
and \eqref{Youngconvolution}. 
Hence, by the real interpolation for multilinear 
operators (see Janson \cite{Janson}),  
it follows that 
if $(q_i)$ satisfy \eqref{scaling} and also 
satisfy 
the strict inequalities 
\begin{equation}\label{strict}
0<\frac{1}{q_i} < 1 - 
\frac{1}{q_0}<1, 
\quad i=1,2,3, 
\end{equation}
then the Lorentz norm estimate 
\begin{equation*} 
\|T(A,B,C)\|_{L^{q'_0, r'_0} (\R^d\times \R^d) }
\lesssim 
\|A\|_{L^{q_1, r_1} (\R^d) }
\|B\|_{L^{q_2, r_2} (\R^d) }
\|C\|_{L^{q_3, r_3} (\R^d) } 
\end{equation*}
holds for all $(r_i)$ such that 
\begin{equation}\label{Holder}
r_i\in [1, \infty], \quad 
i=0, 1, 2, 3, 
\quad \text{and}
\quad 
\frac{1}{r_0} +
\frac{1}{r_1}+ 
\frac{1}{r_2}+
\frac{1}{r_3}
=1. 
\end{equation}
By duality again, 
this implies that 
the inequality 
\begin{equation}\label{BLLorentz}
\begin{split}
&\int_{\R^d \times \R^d} 
V(x_1, x_2)
A(x_1+x_2)
B(x_2)
C(x_2)\, dx_1 dx_2
\\
&\lesssim 
\|V\|_{
L^{q_0,r_0} (\R^{2d}) }
\|A\|_{L^{q_1, r_1} (\R^{d}) }
\|B\|_{L^{q_2, r_2} (\R^{d}) }
\|C\|_{L^{q_3, r_3} (\R^{d}) } 
\end{split}
\end{equation}
holds for all $(q_i)$ and $(r_i)$ satisfying 
\eqref{scaling}, \eqref{strict}, and \eqref{Holder}. 
In particular, by taking 
$q_0=4$, 
$q_1=q_2=q_3=2$, 
$r_0=r_1=\infty$, 
and 
$r_2=r_3=2$, 
we obtain 
\begin{align*}
&
\int_{\R^d \times \R^d} 
V(x_1, x_2)
A(x_1+x_2)
B(x_1)
C(x_2)\, dx_1 dx_2
\\
&
\lesssim 
\|V\|_{
L^{4,\infty}(\R^d\times \R^d) }
\|A\|_{L^{2, \infty} (\R^d)}
\|B\|_{L^{2} (\R^d)}
\|C\|_{L^{2} (\R^d)},  
\end{align*}
which {\it a fortiori\/} implies \eqref{BL4infty222}. 
\end{proof}

\begin{rem} 
The basic idea of using real interpolation 
to derive \eqref{BLLorentz}-\eqref{Holder} 
from 
\eqref{BLq0q1q2q3} 
is given in the paper of 
Perry \cite[Appendix A]{Perry}.  
Theorem A.3 in this Appendix A, 
written by M.\ Christ, 
gives a sufficient condition 
to derive inequality of the form 
\eqref{BLLorentz}-\eqref{Holder}  
from the inequality 
of the form 
\eqref{BLq0q1q2q3}. 
In this general theorem, 
the sufficient condition is expressed in terms 
of $(q_j)$ and subspaces of $\R^{2d}$. 
If $d=1$, then 
by applying this theorem  
we can conclude that 
\eqref{BLLorentz}-\eqref{Holder} holds for 
all $(q_j)$ satisfying \eqref{strict}. 
However, if $d\ge 2$,  
the case 
\eqref{strict} 
does not 
satisfy the very condition of 
the theorem.  
\end{rem}

\begin{rem}
It is also possible to prove 
Proposition \ref{productLweak} 
by the same method as in Proof of Proposition 
\ref{weightL4weak}. 
In fact, by using H\"older's inequality 
and Young's inequality, we see that  
the inequality 
\begin{equation}\label{productweaklemma}
\begin{split}
&\int_{\R^n \times \R^n} 
f_1(x_1) f_2(x_2) 
A(x_1+ x_2) 
B(x_1)
C(x_2)
\, dx_1 dx_2
\\
&\lesssim 
\|f_1\|_{L^{p_1} }
\|f_2\|_{L^{p_2} }
\|A\|_{L^{q_1}}
\|B\|_{L^{q_2}}
\|C\|_{L^{q_3}}
\end{split}
\end{equation}
holds 
for 
\begin{align}
&
1/p_1+1/p_2+1/q_1+1/q_2+1/q_3 
=2, 
\label{HolderYoung1}
\\
&
0\le 1/p_1, 1/p_2, 1/q_1, 1/q_2, 1/q_3 \le 1, 
\label{HolderYoung2}
\\
&
0\le {1}/{q_2}+{1}/{p_1}\le 1, 
\label{HolderYoung3}
\\
&
0\le {1}/{q_3}+{1}/{p_2}\le 1. 
\label{HolderYoung4}
\end{align}
Hence, by the same argument of interpolation as in 
Proof of Proposition 
\ref{weightL4weak}, 
we see that 
\eqref{productweaklemma} holds 
with the Lebesgue norms replaced by 
appropriate Lorentz norms 
if the equality \eqref{HolderYoung1} holds 
and if all the inequalities 
\eqref{HolderYoung2},   
\eqref{HolderYoung3}, and 
\eqref{HolderYoung4} hold with strict inequalities. 
Thus, in particular, 
for $q_1=q_2=q_3=2$ and for $p_1, p_2$ satisfying 
$0<1/p_1, 1/p_2 < 1/2$ and 
$1/p_1+1/p_2=1/2$, 
we have 
\begin{align*}
&
\int_{\R^n \times \R^n} 
f_1(x_1) f_2(x_2) 
A(x_1+ x_2) 
B(x_1)
C(x_2)
\, dx_1 dx_2
\\
&
\lesssim 
\|f_1\|_{ L^{p_1,\infty} }
\|f_2\|_{ L^{p_2, \infty} }
\|A\|_{L^{2, \infty} }
\|B\|_{L^{2}}
\|C\|_{L^{2}}, 
\end{align*}
which {\it a fortiori\/} implies the 
conclusion of Proposition \ref{productLweak}.  
\end{rem}

Here we give a proof of the assertion of 
Example \ref{example-of-V}. 

\begin{proof}[Proof of Example \ref{example-of-V}] 
The function \eqref{V-n/2} is in 
$ \ell^{4, \infty} (\Z^{2n}) $ 
and hence it belongs to 
$\calB (\Z^n \times \Z^n)$ by 
Proposition 
\ref{weightL4weak}. 
The fact that the functions \eqref{Vproduct2} and 
\eqref{Vproduct1} belong to 
$\calB (\Z^n \times \Z^n)$ can be seen 
by the use of 
Propositions  
\ref{productLweak} 
and \ref{propertiesB} (4). 
\end{proof}

We introduce the following.

\begin{defn}\label{defModerate}
Let $d \in \N$. 
We say that a continuous function 
$F: \R^d \to (0, \infty)$   
is of {\it moderate class\/} if  
there exists an $N=N_{F}>0$ such that 
\begin{equation}\label{moderateL2}
F(\xi)^2 \ast \langle \xi \rangle ^{-N} 
=\int_{\R^d} 
F(\eta)^2 
\langle \xi - \eta \rangle ^{-N} 
\, d\eta 
\approx F(\xi)^2 
\; \; \text{for all} \;\; \xi\in \R^d, 
\end{equation}
where the implicit constants in 
$\approx$ may depend on $F$. 
We denote by 
$\calM (\R^d)$ the set of 
all functions on $\R^d$ of moderate class. 
\end{defn}

Here are some simple properties of the class 
$\calM (\R^d)$. 

\begin{prop}\label{moderateproperties}
\begin{enumerate}
\setlength{\itemindent}{0pt} 
\setlength{\itemsep}{3pt} 
\item 
If the relation \eqref{moderateL2} holds for an 
$N>0$, then 
the same relation, possibly with different constants 
in $\approx$, holds if $N$ is replaced by 
$N^{\prime} > \max \{N, d\}$. 
\item 
If $F \in \calM (\R^d)$ and $N>0$ satisfy 
\eqref{moderateL2}, then 
\begin{equation*}
F(\xi) \langle \zeta \rangle^{-N/2} 
\lesssim 
F(\xi+ \zeta)
\lesssim 
F(\xi) \langle \zeta \rangle^{N/2} 
\;\; \text{for all} \;\; 
\xi, \zeta \in \R^d. 
\end{equation*}
\item 
Let $d=d_1 + d_2$ with 
$d_1, d_2\in \N$. 
Then a continuous function  
$F: \R^{d} \to (0, \infty)$ belongs to the class 
$\calM (\R^d)$ if and only if 
the relation 
\begin{equation}\label{moderate-product-weight}
\int_{\R^{d_1}\times \R^{d_2}}
F(\eta_1, \eta_2)^2 
\langle \xi_1- \eta_1 \rangle^{-N_1}
\langle \xi_2- \eta_2 \rangle^{-N_2}
\, d\eta_1 d\eta_2 
\approx 
F(\xi_1, \xi_2)^2
\end{equation}
holds for any sufficiently large $N_1>0$ and $N_2>0$. 
\item 
If $d_1, d_2 \in \N$, 
$F_1 \in \calM (\R^{d_1})$, and 
$F_2 \in \calM (\R^{d_2})$, then 
the function 
$F(\xi_1, \xi_2) 
= F_1 (\xi_1) F_2(\xi_2)$ 
belongs to $\calM (\R^{d_1+d_2})$. 
\end{enumerate}
\end{prop}

\begin{proof}
The assertion (1) follows once 
we make the convolution of the functions in 
\eqref{moderateL2} with 
the function $\langle \xi\rangle^{-N^{\prime}}$ and 
use the fact that 
$\langle \xi \rangle^{-N} \ast 
\langle \xi \rangle^{-N^{\prime}} 
\approx 
\langle \xi \rangle^{-N}$ 
if $N^{\prime}> \max \{N, d\}$.  
The assertion (2) follows from the inequalities 
\[
\langle \xi \rangle ^{-N} 
\langle \zeta \rangle ^{-N} 
\lesssim 
\langle \xi + \zeta \rangle ^{-N} 
\lesssim 
\langle \xi \rangle ^{-N} 
\langle \zeta \rangle ^{N}. 
\]
To prove the assertion (3), 
first observe that if the relation 
\eqref{moderate-product-weight} holds 
then the same relation holds 
if $N_i$ are replaced by 
$N_i^{\prime}>\max \{N_i, d_i\}$, $i=1, 2$. 
This is proved by the same reasoning 
as in the proof of (1). 
Using this fact, the fact of (1), 
and the obvious inequalities 
\[
\langle (\xi_1, \xi_2) \rangle^{- 2N }
\le  
\langle \xi_1\rangle^{-N}
\langle \xi_2 \rangle^{-N}
\le 
\langle (\xi_1, \xi_2) \rangle^{-N} 
\le 
\langle \xi_1\rangle^{-N/2}
\langle \xi_2 \rangle^{-N/2}, 
\]
we can easily prove (3). 
Finally the assertion (4) easily follows 
from (3). 
\end{proof}

Finally we give a general result 
concerning the classes $\calB$ and $\calM$.

\begin{prop} \label{Vast}
For any $V\in \calB ( \Z^d \times \Z^d )$, 
there exists a function 
$V^{\ast} \in \calM (\R^{2d})$ 
such that 
$V(\nu_1, \nu_2) \le V^{\ast} (\nu_1, \nu_2)$ 
for all $(\nu_1, \nu_2) \in \Z^d \times \Z^d$ and 
the restriction of $V^{\ast}$ to $\Z^d \times \Z^d $ 
belongs to $\calB ( \Z^d \times \Z^d )$. 
\end{prop}

\begin{proof}
Suppose 
$V\in \calB ( \Z^d \times \Z^d )$ 
and suppose the inequality 
\eqref{BL222} holds. 
We may assume $V$ is not identically equal to $0$. 
By translation of variables, 
we see that the inequality 
\begin{equation}\label{BL222translation}
\sum_{\nu_1, \nu_2 \in \Z^d} 
V(\nu_1 - \mu_1, \nu_2 - \mu_2) 
A(\nu_1+ \nu_2) 
B(\nu_1)
C(\nu_2)
\le c 
\|A\|_{\ell^2 (\Z^d) } 
\|B\|_{\ell^2 (\Z^d) } 
\|C\|_{\ell^2 (\Z^d) } 
\end{equation} 
holds for all $(\mu_1, \mu_2) \in 
\Z^d \times \Z^d $ with the same constant $c$ as 
in \eqref{BL222}. 
Take a number $N> 2d$. 
Multiplying \eqref{BL222translation} 
by $\langle  (\mu_1, \mu_2)  \rangle ^{-N}$ 
and taking sum over $ (\mu_1, \mu_2)\in \Z^d \times \Z^d $, 
we see that the function 
\begin{align*}
G(\nu_1, \nu_2) 
&
= \sum_{\mu_1, \mu_2 \in \Z^d} 
V(\nu_1 - \mu_1, \nu_2 - \mu_2) 
\langle  (\mu_1, \mu_2)  \rangle ^{-N} 
\\
&= \sum_{\mu_1, \mu_2 \in \Z^d} 
V(\mu_1, \mu_2) 
\langle  ( \nu_1-\mu_1, \nu_2 - \mu_2)  \rangle ^{-N} 
\end{align*} 
also belongs to the class $\calB (\Z^d \times \Z^d)$. 
We shall show that the function 
\[
V^{\ast} (\xi_1, \xi_2) = 
\bigg(
\sum_{\mu_1, \mu_2 \in \Z^d} 
V (\mu_1, \mu_2)^{2} 
\langle  ( \xi_1 - \mu_1, \xi_2 - \mu_2)  \rangle ^{-2N} 
\bigg)^{1/2}, 
\quad (\xi_1, \xi_2) \in \R^d \times \R^d, 
\]
has the desired properties.  
First, $V^{\ast}$ is a positive continuous function on 
$\R^{2d}$. 
For $N^{\prime}> 2N$, we have 
\begin{align*}
&
\int_{\R^{2d}}
V^{\ast} (\xi_1, \xi_2) ^{2} 
\langle (\eta_1 - \xi_1, \eta_2 - \xi_2) \rangle^{ -N^{\prime} } \, 
d\xi_1 d\xi_2
\\
&
=
\int_{\R^{2d}}
\sum_{\mu_1, \mu_2 \in \Z^d} 
V (\mu_1, \mu_2)^{2} 
\langle  ( \xi_1 - \mu_1, \xi_2 - \mu_2)  \rangle ^{-2N} 
\langle (\eta_1 - \xi_1, \eta_2 - \xi_2) \rangle^{ -N^{\prime} } \, 
d\xi_1 d\xi_2
\\
&
\approx 
\sum_{\mu_1, \mu_2 \in \Z^d} 
V (\mu_1, \mu_2)^{2} 
\langle ( \eta_1 - \mu_1, \eta_2 - \mu_2) \rangle^{ -2N } 
= V^{\ast}(\eta_1, \eta_2)^2. 
\end{align*}
Hence $V^{\ast} \in \calM (\R^{2d})$. 
Obviously 
$V^{\ast}(\nu_1, \nu_2) \ge V(\nu_1, \nu_2)$. 
Finally, 
since 
$V^{\ast}(\nu_1, \nu_2) \le G (\nu_1, \nu_2)$ 
(because $\|\cdot \|_{\ell^2} \le \|\cdot \|_{\ell^1}$) 
and since 
$G\in \calB (\Z^d \times \Z^d)$, the restriction of 
$V^{\ast}$ to $\Z^d \times \Z^d $ also belongs to 
$\calB (\Z^d \times \Z^d)$. 
\end{proof}

\section{Main results}\label{sectionMain}

\subsection{Key proposition}
\label{subsectionKeyLemma}

Proposition \ref{main-prop} to be given below  
plays a crucial role in our argument. 
In fact, it already contains the essential 
part of Theorem \ref{main-thm-1} (2)  
and Theorems \ref{main-thm} and 
\ref{main-thm-2} that will be given in 
Subsections \ref{subsectionMainTh} and 
\ref{subsectionvariant}. 
The basic idea of 
the arguments of 
Subsections 
\ref{subsectionKeyLemma}--
\ref{subsectionvariant} 
goes back to 
Boulkhemair 
\cite[Theorem 5]{boulkhemair 1995}.

\begin{prop}
\label{main-prop}
Let $W \in \calM (\R^{2n})$ and suppose 
the restriction of $W$ to $\Z^n \times \Z^n$ belongs 
to the class $\calB (\Z^n \times \Z^n)$. 
For $j = 1, \dots, n$, 
let 
$R_{0,j}, R_{1,j}, R_{2,j} \in [1, \infty)$, 
$1 \leq r_j \leq 2$,
$2 \le p_{1, j}, p_{2, j} \le \infty$, 
and $1/r_j=1/p_{1, j}+1/p_{2, j}$.   
Suppose $\sigma$ is 
a bounded continuous function on 
$(\R^n)^3$ such that  
$\supp \calF \sigma \subset 
\prod_{i=0}^2 ( \prod_{j=1}^n [-R_{i,j},R_{i,j}] )$. 
Then 
\begin{equation}\label{conclusionProp}
\begin{split}
&
\left| \int_{\R^n} 
T_{ \sigma }( f_1, f_2 )(x) 
g(x) 
\, dx \right|
\\
&\lesssim
\bigg(\prod_{j=1}^n 
R_{0, j}^{1/2}
R_{1, j}^{ 1/p_{1, j} }
R_{2, j}^{ 1/p_{2, j} }
\bigg) 
\big\|
W(\xi_1, \xi_2)^{-1} 
\sigma (x,\xi_1,\xi_2) \big\|_{ L^{2}_{ul} ((\R^{n})^3)}
\\
&\quad\times
\| f_1 \|_{L^2} \| f_2 \|_{L^2} \| g \|_{
(L^2, 
\ell^{r_1^\prime} \cdots 
\ell^{r_n^\prime}) }. 
\end{split}
\end{equation}
\end{prop}

\begin{proof}
We rewrite the integral on the 
left hand side of \eqref{conclusionProp}. 
Take a function 
$\kappa \in \calS(\R)$ such that 
$\widehat{\kappa}=1$ on $[-1, 1]$ and 
define the functions $\theta_i$, $i=0, 1, 2$, 
by 
\begin{equation*}
\theta_{i} 
(\zeta_1, \dots, \zeta_n) 
= 
R_{i,1} \cdots R_{i,n} 
\kappa (R_{i,1}\zeta_1) \cdots 
\kappa (R_{i,n}\zeta_n), 
\quad 
(\zeta_1, \dots, \zeta_n) \in \R^n, 
\end{equation*}
Then 
$\widehat {\theta_0} \otimes 
\widehat {\theta_1} \otimes \widehat {\theta_2} = 1$ 
on $\supp \calF \sigma$ 
and hence $\sigma$ can be written as 
\begin{align*}
&
\sigma (x, \xi_1, \xi_2)
=
\int_{(\R^{n})^3} 
\sigma (y, \eta_1, \eta_2)
\theta_0(x-y)
\theta_1(\xi_1-\eta_1)
\theta_2(\xi_2-\eta_2)
\, dyd\eta_1d\eta_2. 
\end{align*}
Thus  
the integral on the left hand side of 
\eqref{conclusionProp} is written as 
\begin{align}
\label{beginning}
\begin{split}
I
&:=
(2\pi)^{2n}
\int_{\R^n}
T_{ \sigma } ( f_1, f_2 )(x)
g (x)
\, dx
\\
&=
\int_{(\R^{n})^6} e^{ i x \cdot (\xi_1+\xi_2)}
\sigma (y, \eta_1, \eta_2)
\\
&\hspace{50pt}\times
\theta_0(x-y) g (x)
\theta_1(\xi_1-\eta_1) \widehat{f_1}(\xi_1) 
\theta_2(\xi_2-\eta_2) \widehat{f_2}(\xi_2)
\, dX, 
\end{split}
\end{align}
where 
$ dX= dx \, d\xi_1 \, d\xi_2 \, 
dy \, d\eta_1 \, d\eta_2$.

Recall that 
$Q = [-1/2, 1/2)^n$ is 
the $n$-dimensional unit cube. 
Since  
$\R^n$ is a disjoint union of the cubes 
$\tau + Q$, $\tau \in \Z^n$, 
integral of a function on $\R^n$ can be written as 
\[
\int_{\R^n} F(x)\, dx = 
\sum_{\tau \in \Z^n} \int_{Q} F(x+ \tau)\, dx. 
\]
By using this formula, we rewrite the integral in 
\eqref{beginning} as 
\begin{align*}
I=
\sum_{ \boldsymbol{\nu},\, 
\boldsymbol{\mu} \in (\Z^{n})^3 }
\int_{Q^{6}} 
&e^{ i (x+\nu_0) 
\cdot (\xi_1+\nu_1+\xi_2+\nu_2)}
\sigma (y+\mu_0, \eta_1+\mu_1, \eta_2+\mu_2)
\\
&\times
\theta_0(x+\nu_0-y-\mu_0) g(x+\nu_0)  
\\
&\times
\theta_1(\xi_1+\nu_1-\eta_1-\mu_1) 
\widehat{f_1}(\xi_1+\nu_1)
\\
&\times
\theta_2(\xi_2+\nu_2-\eta_2-\mu_2) 
\widehat{f_2}(\xi_2+\nu_2) 
	\, dX, 
\end{align*}
where 
$\boldsymbol{\nu} 
=(\nu_0,\nu_1,\nu_2), 
\boldsymbol{\mu} 
= (\mu_0,\mu_1,\mu_2) \in (\Z^{n})^3$.

We rewrite the exponential term as 
\begin{align*}
e^{ i (x+\nu_0) \cdot (\xi_1+\nu_1+\xi_2+\nu_2)}
=
e^{ i (\nu_1+\nu_2) \cdot x}
e^{ i \nu_0 \cdot \xi_1}
e^{ i \nu_0 \cdot \xi_2}
e^{ i \nu_0 \cdot (\nu_1+\nu_2)}
\sum_{ 
\alpha = \beta + \gamma } 
\frac{i^{|\alpha|}}{\beta! \gamma! }
x^{\alpha} \xi_1^{\beta} \xi_2^{\gamma}.
\end{align*}
Now the variables $x, \xi_1, \xi_2$ 
are separated and 
$I$ is written as 
\begin{align*}
I=
\sum_{ \boldsymbol{\nu},\, \boldsymbol{\mu} \in (\Z^{n})^3 }
&
\sum_{ 
\alpha = \beta + \gamma } 
\frac{i^{|\alpha|}}{\beta! \gamma!}
e^{ i \nu_0 \cdot (\nu_1+\nu_2)} 
\int_{Q^{3}}	
\sigma (y+\mu_0, \eta_1+\mu_1, \eta_2+\mu_2)
\\ \times \Bigg( 
&\int_{Q} e^{ i (\nu_1+\nu_2) \cdot x} 
	\theta_0(x+\nu_0-y-\mu_0) g(x+\nu_0) 
    x^{\alpha}\, dx \Bigg)
\\ \times \Bigg( 
&\int_{Q} e^{ i \nu_0 \cdot \xi_1} 
	\theta_1(\xi_1+\nu_1-\eta_1-\mu_1) 
\widehat{f_1}(\xi_1+\nu_1) \xi_1^{\beta}
	d\xi_1 \Bigg)
\\ \times \Bigg( 
&\int_{Q} e^{ i \nu_0 \cdot \xi_2} 
	\theta_2(\xi_2+\nu_2-\eta_2-\mu_2) 
\widehat{f_2}(\xi_2+\nu_2) \xi_2^{\gamma}
	d\xi_2 \Bigg)
\, dy d\eta_1 d\eta_2.
\end{align*}

We take a sufficiently large even positive integer $N$. 
Then, since $\langle z \rangle^N$ is 
a polynomial of $z$ of order $N$, we can write 
\begin{equation}\label{anglez}
\langle \nu_0 - \mu_0 \rangle^N
=
\sum_{ 
|\alpha_1+ \alpha_2+ \alpha_3|\leq N } 
C_{\alpha_1,\alpha_2,\alpha_3} 
(x+\nu_0 - y - \mu_0)^{\alpha_1} 
x^{\alpha_2} y^{\alpha_3} 
\end{equation}
and hence 
\begin{align*}
&\theta_0 (x+\nu_0-y-\mu_0) 
\\
&=
\langle \nu_0 - \mu_0 \rangle ^{-N}
\theta_0 (x+\nu_0-y-\mu_0)
\sum_{ 
|\alpha_1+ \alpha_2 +\alpha_3|\leq N } 
C_{\alpha_1,\alpha_2,\alpha_3} 
(x+\nu_0 - y - \mu_0)^{\alpha_1}  
x^{\alpha_2} y^{\alpha_3} 
\\
&=
\langle \nu_0 - \mu_0 \rangle ^{-N}
\sum_{ 
|\alpha_1+ \alpha_2+ \alpha_3|\leq N }  
C_{\alpha_1,\alpha_2,\alpha_3} 
\widetilde{\theta_0^{\alpha_1}}(x+\nu_0-y-\mu_0) 
x^{\alpha_2} y^{\alpha_3},   
\end{align*}
where 
$
\widetilde{\theta_0^{\alpha_1}}(z)
=
\theta_0 (z) z^{\alpha_1}$. 
We also rewrite the 
$\theta_1 (\dots)$ and $\theta_2 (\dots)$ 
in the same way. 
Thus we obtain 
\begin{align*}
I&=
\sum_{ 
\alpha = \beta + \gamma } \, 
\sum_{ |\alpha_1 + \alpha_2 +\alpha_3| \leq N }\, 
\sum_{ |\beta_1+ \beta_2+\beta_3| \leq N }\, 
\sum_{ |\gamma_1+ \gamma_2+ \gamma_3| \leq N }\, 
\frac{i^{|\alpha|}
C_{\boldsymbol{\alpha},\boldsymbol{\beta},\boldsymbol{\gamma}}
}{\beta! \gamma!}\, 
e^{i \nu_0 \cdot (\nu_1 + \nu_2)}
\\
& 
\times 
\sum_{ \boldsymbol{\nu},\, \boldsymbol{\mu} \in (\Z^{n})^3 }
\int_{Q^{3}}
	\sigma (y+\mu_0, \eta_1+\mu_1, \eta_2+\mu_2) 
y^{\alpha_3} \eta_1^{\beta_3} \eta_2^{\gamma_3}
\\
&
\qquad \times 
\langle \nu_0 - \mu_0 \rangle^{-N}
\langle \nu_1 - \mu_1 \rangle^{-N}
\langle \nu_2 - \mu_2 \rangle^{-N}
\\
&\qquad\times
\Bigg( \int_{Q} e^{ i (\nu_1+\nu_2) \cdot x} \, 
	\widetilde{\theta_0^{\alpha_1}}(x+\nu_0-y-\mu_0) 
g(x+\nu_0)  x^{\alpha+\alpha_2}
	\, dx \Bigg)
\\
&\qquad\times
\Bigg( \int_{Q} e^{ i \nu_0 \cdot \xi_1} \, 
	\widetilde{\theta_1^{\beta_1}}
(\xi_1+\nu_1-\eta_1-\mu_1) 
	\widehat{f_1}(\xi_1+\nu_1) \xi_1^{\beta+\beta_2}
	\, d\xi_1 \Bigg)
\\
&\qquad\times
\Bigg( \int_{Q} e^{ i \nu_0 \cdot \xi_2} \, 
	\widetilde{\theta_2^{\gamma_1}}
(\xi_2+\nu_2-\eta_2-\mu_2) 
	\widehat{f_2}(\xi_2+\nu_2) \xi_2^{\gamma+\gamma_2}
	\, d\xi_2 \Bigg)
\, dy d\eta_1 d\eta_2
\\
&
= \sum_{\boldsymbol{\alpha},\boldsymbol{\beta},\boldsymbol{\gamma}}
\frac{i^{|\alpha|}
C_{\boldsymbol{\alpha},\boldsymbol{\beta},\boldsymbol{\gamma}}
}{\beta! \gamma!}\, 
e^{i \nu_0 \cdot (\nu_1 + \nu_2)}
I_{\boldsymbol{\alpha},\boldsymbol{\beta},\boldsymbol{\gamma}}, 
\end{align*}
where 
$C_{\boldsymbol{\alpha},\boldsymbol{\beta},\boldsymbol{\gamma}}
=C_{\alpha_1, \alpha_2, \alpha_3} 
C_{\beta_1, \beta_2, \beta_3} 
C_{\gamma_1, \gamma_2, \gamma_3}$ 
is the product of the constants in \eqref{anglez}, 
$I_{\boldsymbol{\alpha},\boldsymbol{\beta},\boldsymbol{\gamma}}$ 
denotes the part 
$\sum_{\boldsymbol{\nu}, 
\boldsymbol{\mu}} \int_{Q^3} 
\dots dyd\eta_1 d \eta_2$ of the formula, 
and 
$\boldsymbol{\alpha} = (\alpha, \alpha_1, \alpha_2, \alpha_3)$, 
$\boldsymbol{\beta} 
= (\beta, \beta_1, \beta_2, \beta_3)$, 
$\boldsymbol{\gamma} = 
(\gamma, \gamma_1, \gamma_2, \gamma_3)$.

Now we shall estimate $I$. 
Notice that in the last expression of $I$, 
the sums over 
$\alpha_i, \beta_i, \gamma_i$ are  
taken over finite sets and 
the sum over 
$\alpha, \beta, \gamma \in (\N_0)^n$, 
$\alpha=\beta+\gamma$,   
has the factor $1/(\beta! \gamma!) $.  
Hence, in order to prove the estimate 
for $I$,  
it is sufficient to show that 
$I_{\boldsymbol{\alpha}, 
\boldsymbol{\beta}, 
\boldsymbol{\gamma}}$ 
is bounded by the right hand side of 
\eqref{conclusionProp}
uniformly in 
$\alpha, \beta, \gamma \in (\N_0)^n$.

Using the obvious estimate 
$|y^{\alpha_3} \eta_1^{\beta_3} \eta_2^{\gamma_3}| \leq 1$ 
for $y, \eta_1, \eta_2 \in Q$ 
and using 
the Cauchy--Schwarz inequality  
with respect to the integral over $y,\eta_1,\eta_2$, 
we obtain 
\begin{align}
\label{beforepeetre}
\begin{split}
&
\big|
I_{\boldsymbol{\alpha},\boldsymbol{\beta},\boldsymbol{\gamma}}
\big|
\\
&\leq 
\sum_{ \boldsymbol{\nu},\, \boldsymbol{\mu} \in (\Z^{n})^3 }
W (\mu_1, \mu_2)^{-1}
\| \sigma (y+\mu_0, \eta_1+\mu_1, \eta_2+\mu_2) 
\|_{L^2_{y,\eta_1,\eta_2}(Q^3)}
\\
&\quad\times
W (\mu_1, \mu_2) 
\langle \nu_0 - \mu_0 \rangle^{-N}
\langle \nu_1 - \mu_1 \rangle^{-N}
\langle \nu_2 - \mu_2 \rangle^{-N}
\\
&\quad 
\times 
\left\|\int_{Q} e^{ i (\nu_1+\nu_2) \cdot x} 
\widetilde{\theta_0^{\alpha_1}}(x+\nu_0-y-\mu_0) 
g(x+\nu_0) x^{\alpha+\alpha_2}
\, dx \right\|_{L^2_{y}(Q)}
\\
&\quad\times
\left\| \int_{Q} e^{ i \nu_0 \cdot \xi_1} 
\widetilde{\theta_1^{\beta_1}}
(\xi_1+\nu_1-\eta_1-\mu_1) 
\widehat{f_1}(\xi_1+\nu_1) \xi_1^{\beta+\beta_2}
\, d\xi_1 \right\|_{L^2_{\eta_1}(Q)}
\\
&\quad\times
\left\| \int_{Q} e^{ i \nu_0 \cdot \xi_2} 
\widetilde{\theta_2^{\gamma_1}}
(\xi_2+\nu_2-\eta_2-\mu_2) 
\widehat{f_2}(\xi_2+\nu_2) 
\xi_2^{\gamma+\gamma_2}
\, d\xi_2 \right\|_{L^2_{\eta_2}(Q)}.
\end{split}
\end{align}

By virtue of the properties of the 
moderate function $W$ as given 
in Proposition \ref{moderateproperties}, (2) and (3), 
we have 
\begin{equation}\label{vsigmaL2ul}
\begin{split}
&
\sup_{\mu_0, \mu_1, \mu_2}
\big\{
W (\mu_1, \mu_2)^{-1}
\| \sigma (y+\mu_0,\eta_1+\mu_1,\eta_2+\mu_2) 
\|_{ L^{2}_{y, \eta_1, \eta_2}(Q^3) }
\big\}
\\
&\approx 
\|
W (\xi_1, \xi_2)^{-1} 
\sigma (x, \xi_1, \xi_2)
\|_{L^2_{ul} ((\R^n)^3)} 
\end{split} 
\end{equation}
and 
\begin{equation}\label{moderate2}
\big\|
W (\mu_1, \mu_2)  
\langle \nu_0 - \mu_0 \rangle^{-N}
\langle \nu_1 - \mu_1 \rangle^{-N}
\langle \nu_2 - \mu_2 \rangle^{-N}
\big\|
_{
\ell^2_{\mu_0} 
\ell^2_{\mu_1}
\ell^2_{\mu_2}
}
\approx 
W (\nu_1, \nu_2)  
\end{equation}
if $N$ is chosen sufficiently large. 
Hence, applying the Cauchy--Schwarz inequality 
to the sum over 
$\boldsymbol{\mu}=(\mu_0, \mu_1, \mu_2)$ 
in \eqref{beforepeetre}, 
and 
using \eqref{vsigmaL2ul} and 
\eqref{moderate2},  
we obtain 
\begin{equation}
\label{beforeABC}
\begin{split}
\big|
I_{\boldsymbol{\alpha},\boldsymbol{\beta},\boldsymbol{\gamma}}
\big|
&\lesssim 
\|
W (\xi_1, \xi_2)^{-1} 
\sigma (x, \xi_1, \xi_2)
\|_{L^2_{ul} ((\R^n)^3) }
\, \sum_{ \nu_0, \nu_1, \nu_2  \in \Z^{n} }
W (\nu_1, \nu_2) 
\\
&\quad\times
\left\|\int_{Q} e^{ i (\nu_1+\nu_2) \cdot x} 
\widetilde{\theta_0^{\alpha_1}}(x+\nu_0-y-\mu_0) 
g(x+\nu_0) x^{\alpha+\alpha_2}
\, dx\right\|_{L^2_{y}(Q) \ell^2_{\mu_0}}
\\
&\quad\times
\left\| \int_{Q} e^{ i \nu_0 \cdot \xi_1} 
\widetilde{\theta_1^{\beta_1}}(\xi_1+\nu_1-\eta_1-\mu_1) 
\widehat{f_1}(\xi_1+\nu_1) \xi_1^{\beta+\beta_2}
\, d\xi_1 \right\|_{L^2_{\eta_1}(Q) \ell^2_{\mu_1}}
\\
&\quad\times
\left\| \int_{Q} e^{ i \nu_0 \cdot \xi_2} 
\widetilde{\theta_2^{\gamma_1}}(\xi_2+\nu_2-\eta_2-\mu_2) 
\widehat{f_2}(\xi_2+\nu_2) \xi_2^{\gamma+\gamma_2}
\, d\xi_2 \right\|_{L^2_{\eta_2}(Q) \ell^2_{\mu_2}}.
\end{split}
\end{equation}
In what follows, we will simply write
\begin{align*}
A_{\boldsymbol{\alpha}}(\nu_3,\nu_0)
&=
\left\|\int_{Q} e^{ i \nu_3 \cdot x} 
\widetilde{\theta_0^{\alpha_1}}(x+\nu_0-y-\mu_0) 
g(x+\nu_0) x^{\alpha+\alpha_2}
\, dx\right\|_{L^2_{y}(Q) \ell^2_{\mu_0}}, 
\\
B_{\boldsymbol{\beta}} (\nu_0,\nu_1)
&=
\left\| \int_{Q} e^{ i \nu_0 \cdot \xi_1} 
\widetilde{\theta_1^{\beta_1}}(\xi_1+\nu_1-\eta_1-\mu_1) 
\widehat{f_1}(\xi_1+\nu_1) \xi_1^{\beta+\beta_2}
\, d\xi_1 \right\|_{L^2_{\eta_1}(Q) \ell^2_{\mu_1}}, 
\\
C_{\boldsymbol{\gamma}}(\nu_0,\nu_2)
&=
\left\| \int_{Q} e^{ i \nu_0 \cdot \xi_2} 
\widetilde{\theta_2^{\gamma_1}}(\xi_2+\nu_2-\eta_2-\mu_2) \widehat{f_2}(\xi_2+\nu_2) \xi_2^{\gamma+\gamma_2}
\, d\xi_2 \right\|_{L^2_{\eta_2}(Q) \ell^2_{\mu_2}}.  
\end{align*}
Thus the inequality 
\eqref{beforeABC} is written as 
\begin{equation}\label{IleqII}
\big|
I_{\boldsymbol{\alpha},\boldsymbol{\beta},\boldsymbol{\gamma}}
\big|
\lesssim 
\|
W (\xi_1, \xi_2)^{-1} 
\sigma (x, \xi_1, \xi_2)
\|_{L^2_{ul} ((\R^n)^3) } 
\II_{\boldsymbol{\alpha},\boldsymbol{\beta},\boldsymbol{\gamma}}  
\end{equation}
with 
\begin{equation}\label{IIalphabetagamma}
\II_{\boldsymbol{\alpha},\boldsymbol{\beta},\boldsymbol{\gamma}}
=
\sum_{ \nu_0, \nu_1, \nu_2  \in \Z^{n}  }
W (\nu_1, \nu_2) 
A_{\boldsymbol{\alpha}} (\nu_1+\nu_2,\nu_0)
B_{\boldsymbol{\beta}} (\nu_0,\nu_1)
C_{\boldsymbol{\gamma}} (\nu_0,\nu_2). 
\end{equation}
We shall estimate 
$\II_{\boldsymbol{\alpha},\boldsymbol{\beta},\boldsymbol{\gamma}}$.

To the sum over $\nu_1, \nu_2$ in 
\eqref{IIalphabetagamma}, 
we apply the $\ell^2$ estimate 
assured by our assumption that  
$W$ restricted to $\Z^n \times \Z^n$ 
belongs to the class $\calB (\Z^n \times \Z^n)$  
to obtain  
\begin{equation*}
\II_{\boldsymbol{\alpha},\boldsymbol{\beta},\boldsymbol{\gamma}}
\lesssim 
\sum_{ \nu_0 \in \Z^{n} }
\| A_{\boldsymbol{\alpha}} (\nu_3,\nu_0) \|_{\ell^2_{\nu_3}}
\| B_{\boldsymbol{\beta}} (\nu_0,\nu_1) \|_{\ell^2_{\nu_1}}
\| C_{\boldsymbol{\gamma}} (\nu_0,\nu_2) \|_{\ell^2_{\nu_2}}. 
\end{equation*}
To estimate the sum over 
$\nu_0=(\nu_{0,1}, \dots, \nu_{0,n})\in \Z^n$, 
we use the H\"older inequality with 
the exponents $1=1/r'_j + 1/p_{1,j} + 1/p_{2,j}$.  
Thus we have 
\begin{equation}\label{IIHolder}
\begin{split}
&\II_{\boldsymbol{\alpha},
\boldsymbol{\beta},
\boldsymbol{\gamma}}
\\
&\lesssim
\| A_{\boldsymbol{\alpha}} (\nu_3,\nu_0) 
\|_{
\ell^2_{\nu_3} 
\ell^{r'_1}_{ \nu_{0, 1} }
\dots 
\ell^{r'_n}_{ \nu_{0, n} }
}
\| B_{\boldsymbol{\beta}} (\nu_0,\nu_1) 
\|_{
\ell^2_{\nu_1} 
\ell^{p_{1,1}}_{ \nu_{0, 1} } \dots 
\ell^{p_{1, n}}_{ \nu_{0, n} }
}
\| C_{\boldsymbol{\gamma}} (\nu_0,\nu_2) 
\|_{
\ell^2_{\nu_2} 
\ell^{p_{2,1}}_{ \nu_{0, 1} } \dots 
\ell^{p_{2,n}}_{ \nu_{0, n} }
}.
\end{split}
\end{equation}

The norm of $A_{\boldsymbol{\alpha}}$ in \eqref{IIHolder} 
is estimated by the use of the Parseval identity 
in $\ell^2_{\nu_3}$ as follows: 
\begin{align*}
&
\| A_{\boldsymbol{\alpha}} (\nu_3,\nu_0) 
\|_{
\ell^2_{\nu_3} 
\ell^{r'_1}_{ \nu_{0, 1} }
\dots 
\ell^{r'_n}_{ \nu_{0, n} }
}
\\
&=
\left\|\int_{Q} e^{ i \nu_3 \cdot x} 
\widetilde{\theta_0^{\alpha_1}}(x+\nu_0-y-\mu_0) 
g(x+\nu_0) x^{\alpha+\alpha_2}
\, dx
\right\|_{
L^2_y (Q) 
\ell^2_{\mu_0} 
\ell^2_{\nu_3} 
\ell^{r'_1}_{ \nu_{0, 1} }
\dots 
\ell^{r'_n}_{ \nu_{0, n} }
}
\\
&\approx
\big\|
\widetilde{\theta_0^{\alpha_1}}(x+\nu_0-y-\mu_0) 
g(x+\nu_0) x^{\alpha+\alpha_2}
\big\|_{
L^2_y (Q) 
\ell^2_{\mu_0} 
L^2_x (Q) 
\ell^{r'_1}_{ \nu_{0, 1} }
\dots 
\ell^{r'_n}_{ \nu_{0, n} }
}
\\
&\leq 
\big\|
\widetilde{\theta_0^{\alpha_1}}(x+\nu_0-y) 
g(x+\nu_0)
\big\|_{
L^2_y(\R^n) 
L^2_x(Q) 
\ell^{r'_1}_{ \nu_{0, 1} }
\dots 
\ell^{r'_n}_{ \nu_{0, n} }
}
\\
&= 
\| \widetilde{\theta_0^{\alpha_1}} \|_{L^2(\R^n)}
\| g \|_{
(L^2, 
\ell^{r'_1} 
\dots 
\ell^{r'_n})
}, 
\end{align*}
where the inequality 
$\le $ on the fourth line holds because 
$|x^{\alpha+\alpha_2}| \leq 1$ for $x \in Q$ and 
$\| F (y+\mu_0) \|_{L^2_y (Q) \ell^2_{\mu_0}} 
= \| F \|_{L^2(\R^n)}$. 
Recall that $\widetilde{\theta_0^{\alpha_1}} (y) $ is 
defined by 
\begin{equation}\label{theta0alpha1}
\begin{split}
&\widetilde{\theta_0^{\alpha_1}} (y) 
\\
&=R_{0,1}^{1 - \alpha_{1,1}}
\cdots 
R_{0,n}^{1 - \alpha_{1,n}}
\kappa (R_{0,1} y_1) 
\cdots 
\kappa (R_{0,n} y_n)  
(R_{0,1} y_1)^{\alpha_{1,1}}
\cdots 
(R_{0,n} y_n)^{\alpha_{1,n}}. 
\end{split}
\end{equation}
Thus, since a function of the form 
$\kappa (z) z^{\alpha}$ 
belongs to the Schwartz class 
$\calS (\R)$ 
and since $R_{0,j} \ge 1$, we have 
\begin{equation}\label{theta0}
\| \widetilde{\theta_0^{\alpha_1}} 
\|_{L^2(\R^n)}
\approx 
\prod_{j=1}^{n} 
R_{0, j}^{1/2-\alpha_{1, j }}
\le  
\prod_{j=1}^{n} 
R_{0, j}^{ 1/2 }. 
\end{equation}
Therefore 
\begin{equation}\label{estimateA}
\| A_{\boldsymbol{\alpha}} (\nu_3,\nu_0) 
\|_{
\ell^2_{\nu_3} 
\ell^{r'_1}_{ \nu_{0,1} }
\dots 
\ell^{r'_n}_{ \nu_{0,n} }
}
\lesssim 
\bigg(
\prod_{j=1}^{n} 
R_{0, j}^{ 1/2 }
\bigg)  
\| g 
\|_{
(L^2, 
\ell^{r'_1}
\dots 
\ell^{r'_n} 
) 
}.
\end{equation}

For the norm of $B_{\boldsymbol{\beta}}$ in 
\eqref{IIHolder}, 
we use \eqref{mixedMinkowski},  
the Hausdorff--Young inequality 
for $\ell^{p_{1,1}}_{\nu_{0,1}} $, 
and the inequality 
$\|\cdot \|_{\ell^2} \le \|\cdot \|_{ \ell^{p'_{1,1}} } $ 
to obtain 
\begin{align*}
&
\| B_{ \boldsymbol{\beta} } (\nu_0,\nu_1) 
\|_{
\ell^2_{ \nu_1 } 
\ell^{ p_{1,1} }_{ \nu_{0, 1} }
\dots 
\ell^{ p_{1,n} }_{ \nu_{0, n} }
}
\\
&=
\left\|
\int_{Q} 
e^{ i \nu_0 \cdot \xi_1 } 
\widetilde{ \theta_1^{\beta_1} }
( \xi_1+\nu_1-\eta_1-\mu_1 ) 
\widehat{f_1} ( \xi_1+\nu_1 ) 
\xi_1^{ \beta +\beta_2 }
\, d\xi_1 
\right\|
_{
L^2_{ \eta_1 } (Q) 
\ell^2_{ \mu_1 } 
\ell^2_{ \nu_1 } 
\ell^{ p_{1,1} }_{ \nu_{0,1} } \dots 
\ell^{ p_{1,n} }_{ \nu_{0,n} }
}
\\
&
\le 
\left\|\dots 
\right\|_{
\ell^{ p_{1,1} }_{ \nu_{0, 1} } 
L^2_{ \eta_1 } (Q) 
\ell^2_{ \mu_1 } 
\ell^2_{ \nu_1 } 
\ell^{ p_{1,2} }_{ \nu_{0, 2} }
\dots 
\ell^{ p_{1,n} }_{ \nu_{0, n} }
 }
\\
&
\lesssim 
\bigg\|
\int_{Q^{n-1}} 
e^{ i (\nu_{0,2}, \dots, \nu_{0,n})\cdot 
(\xi_{1,2}, \dots, \xi_{1,n}) }
\widetilde{ \theta_1^{\beta_1} }
( \xi_1+\nu_1-\eta_1-\mu_1 ) 
\\
&\qquad \qquad 
\widehat{f_1} ( \xi_1+\nu_1 ) 
\xi_1^{ \beta +\beta_2 }
\, d\xi_{1,2} \dots d\xi_{1,n}
\bigg\|
_{
L^{ p'_{1,1} }_{ \xi_{1,1} }(I)  
L^2_{ \eta_1 } (Q) 
\ell^2_{ \mu_1 } 
\ell^2_{ \nu_1 } 
\ell^{ p_{1,2} }_{ \nu_{0, 2} }
\dots 
\ell^{ p_{1,n} }_{ \nu_{0, n} }
}
\\
&
\le  
\big\|
\dots 
\big\|
_{
L^{ p'_{1,1} }_{ \xi_{1,1} }(I) 
\ell^{ p'_{1,1} }_{ \nu_{1,1} }
L^2_{ \eta_1 } (Q) 
\ell^2_{ \mu_1 } 
\ell^2_{ \nu_{1,2} } 
\dots 
\ell^2_{ \nu_{1,n} } 
\ell^{ p_{1,2} }_{ \nu_{0, 2} }
\dots 
\ell^{ p_{1,n} }_{ \nu_{0, n} }
},  
\end{align*}
where $I =[-1/2, 1/2)$. 
We then repeat the same arguments for 
$\ell^{p_{1,2}}_{\nu_{0,2}} , \dots, \ell^{p_{1,n}}_{\nu_{0,n}} $ 
in  this order to obtain
\begin{align*}
&
\|B_{ \boldsymbol{\beta} } (\nu_0,\nu_1)
\|
_{
\ell^2_{\nu_{1}} 
\ell^{p_{1,1}}_{\nu_{0,1}} 
\cdots 
\ell^{p_{1,n}}_{\nu_{0,n}} 
}
\\
&\lesssim
\big\| 
\widetilde{\theta_1^{\beta_1}}
( \xi_1+\nu_1-\eta_1-\mu_1 ) 
\widehat{f_1}( \xi_1+\nu_1 )
\big\|
_{ 
L^{ p'_{1,n} } _{ \xi_{1,n} } (I) 
\ell^{ p'_{1,n} }_{ \nu_{1,n} }
\dots 
L^{ p'_{1,1} } _{\xi_{1,1}} (I)  
\ell^{ p'_{1,1} }_{ \nu_{1,1} }
L^2 _{\eta_1} (Q) \ell^2_{\mu_1} 
}
\\
&
=
\big\| 
\widetilde{\theta_1^{\beta_1}}
( \xi_1-\eta_1 ) 
\widehat{f_1}( \xi_1 )
\big\|
_{ 
L^{ p'_{1,n} } _{ \xi_{1,n} } (\R) 
\dots 
L^{ p'_{1,1} } _{\xi_{1,1}} (\R)  
L^2 _{\eta_1} (\R^n)
}. 
\end{align*} 
Changing variables 
$\xi_{1, j} \to \xi_{1, j} + \eta_{1, j}$ for  
$j=1, \dots, n$, 
and using \eqref{mixedMinkowski}, 
we have 
\begin{align*}
&
\big\| 
\widetilde{\theta_1^{\beta_1}}
( \xi_1-\eta_1 ) 
\widehat{f_1}( \xi_1 )
\big\|
_{ 
L^{ p'_{1,n} } _{ \xi_{1,n} } (\R) 
\dots 
L^{ p'_{1,1} } _{\xi_{1,1}} (\R)  
L^2 _{\eta_1} (\R^n)
}
\\
&
=
\big\| 
\widetilde{\theta_1^{\beta_1}}
( \xi_1 ) 
\widehat{f_1}( \xi_1 + \eta_1 )
\big\|
_{ 
L^{ p'_{1,n} } _{ \xi_{1,n} } (\R) 
\dots 
L^{ p'_{1,1} } _{\xi_{1,1}} (\R)  
L^2 _{\eta_1} (\R^n)
}
\\
&\le 
\big\| 
\widetilde{\theta_1^{\beta_1}}
( \xi_1 ) 
\widehat{f_1}( \xi_1 + \eta_1 )
\big\|
_{ 
L^2 _{\eta_1} (\R^n)
L^{ p'_{1,n} } _{ \xi_{1,n} } (\R) 
\dots 
L^{ p'_{1,1} } _{\xi_{1,1}} (\R)  
}
\\
&=
\big\| 
\widetilde{\theta_1^{\beta_1}} ( \xi_1)
\big\|
_{ 
L^{ p'_{1,n} } _{ \xi_{1,n} } (\R) 
\dots 
L^{ p'_{1,1} } _{\xi_{1,1}} (\R)  
}
\|\widehat{f_1}\|_{L^2}. 
\end{align*}
For the mixed norm of 
$\widetilde{\theta_1^{\beta_1}} $ 
in the last expression, 
by the same reason as we deduced 
\eqref{theta0} from \eqref{theta0alpha1}, 
we have  
\[
\big\| 
\widetilde{\theta_1^{\beta_1}} ( \xi_1)
\big\|
_{ 
L^{ p'_{1,n} } _{ \xi_{1,n} } (\R) 
\dots 
L^{ p'_{1,1} } _{\xi_{1,1}} (\R)  
}
\lesssim 
\prod_{j=1}^{n} R_{1, j}^{ 1- 1/p'_{1,j} }
=
\prod_{j=1}^{n} R_{1, j}^{ 1/p_{1,j} }. 
\]
Also 
$\|\widehat{f_1}\|_{L^2}\approx 
\|f_1\|_{L^2}$ 
by Plancherel's theorem. 
Now combining the inequalities obtained above, 
we get 
\begin{equation}\label{estimateB}
\|B_{ \boldsymbol{\beta} }  (\nu_0,\nu_1)
\|
_{
\ell^2_{\nu_{1}} 
\ell^{p_{1,1}}_{\nu_{0,1}} 
\cdots 
\ell^{p_{1,n}}_{\nu_{0,n}} 
}
\lesssim 
\bigg(
\prod_{j=1}^{n} R_{1, j}^{ 1/p_{1,j} } 
\bigg)
\|f_1\|_{L^2}. 
\end{equation}

Similarly, we have
\begin{equation}\label{estimateC}
\|C_{ \boldsymbol{\gamma} }  (\nu_0,\nu_2)
\|
_{
\ell^2_{ \nu_{2} } 
\ell^{ p_{2,1} }_{ \nu_{0,1} } 
\cdots 
\ell^{ p_{2,n} }_{ \nu_{0,n} }
}
\lesssim 
\bigg( 
\prod_{j=1}^{n} 
R_{ 2, j }^{ 1/p_{ 2, j} } 
\bigg)
\| f_2 \|_{L^2}.
\end{equation}

The desired inequality \eqref{conclusionProp} 
now follows from 
\eqref{IleqII}, 
\eqref{IIHolder}, 
\eqref{estimateA}, 
\eqref{estimateB}, and 
\eqref{estimateC}. 
This completes the proof of Proposition 
\ref{main-prop}. 
\end{proof}

\subsection{A theorem for symbols with 
limited smoothness}
\label{subsectionMainTh}

From Proposition \ref{main-prop}, 
we shall deduce a theorem concerning 
bilinear pseudo-differential 
operators $T_{\sigma}$ with 
symbols of limited smoothness. 
To measure the smoothness of such symbols, 
we shall use Besov type norms.  
To define the Besov type norms, 
we use the partition of unity given as follows. 
Let $d\in \N$. 
Take a $\phi \in \calS (\R^d)$ 
such that $\phi (\xi)= 1$ for 
$ | \xi | \leq 1 $  
and 
$\supp \phi \subset 
\{ \xi \in \R^d : | \xi | \leq 2 \}$.  
We put $\psi (\xi)= \phi (\xi) - \phi (2\xi)$. 
Then 
$\supp \psi \subset 
\{ \xi \in \R^d : 1/2 \leq | \xi | \leq 2 \}$.
We set 
$\psi_0 = \phi $ and 
$\psi_k = \psi (\cdot / 2^k )$ for $k \in \N$.
Then $\sum_{k=0}^{\infty} \psi_{k} (\xi) = 1$ 
for all $\xi \in \R^d$. 
We shall call $\{\psi_k\}_{k\in \N_0}$ a  
Littlewood--Paley partition of unity on $\R^d$.  
It is easy to see that 
the Besov type norms given in the following definition 
do not depend, up to the equivalence of norms, 
on the choice of Littlewood--Paley partition of unity.

\begin{defn}\label{defBSWs} 
Let $W\in \calM (\R^{2n})$. 
Let $\{\psi_{k}\}_{ k\in \N_0 }$ 
be a Littlewood--Paley 
partition of unity on $\R$.  
For 
\begin{align*}
&
\boldsymbol{s}= 
(s_{0,1}, \dots, s_{0,n}, 
s_{1,1}, \dots, s_{1,n}, 
s_{2,1}, \dots, s_{2,n})
\in [0, \infty)^{3n}, 
\\
&
\boldsymbol{k}= 
(k_{0,1}, \dots, k_{0,n}, 
k_{1,1}, \dots, k_{1,n}, 
k_{2,1}, \dots, k_{2,n})
\in (\N_0)^{3n}, 
\end{align*}
and $\sigma = \sigma (x, \xi_1, \xi_2) 
\in L^{\infty} 
(\R^n_{x}\times 
\R^n_{\xi_1}\times 
\R^n_{\xi_2})$, 
we write 
$\boldsymbol{s}\cdot 
\boldsymbol{k}
=
\sum_{i=0}^{2}
\sum_{j=1}^{n}
s_{i,j} k_{i,j}$ and 
\begin{equation*}
\Delta_{\boldsymbol{k}} \sigma (x, \xi_1, \xi_2)
=
\bigg( \prod_{j=1}^{n} 
\psi_{k_{0, j} } ( D_{ x_j } ) 
\psi_{k_{1, j} } ( D_{ \xi_{1, j} } ) 
\psi_{k_{2, j} } ( D_{ \xi_{2,j} } ) 
\bigg) 
\sigma (x, \xi_1, \xi_2). 
\end{equation*}
We denote by 
$BS^{W}_{0,0} (\boldsymbol{s}; \R^n) $ 
the set of all $\sigma \in L^{\infty} ((\R^n)^3)$ for which 
the following norm is finite: 
\begin{equation*} 
\|\sigma\|_{ 
BS^{W}_{0,0} (\boldsymbol{s}; \R^n ) 
} 
= 
\sum_{\boldsymbol{k} \in (\N_0)^{3n} }\, 
2^{ 
\boldsymbol{s}\cdot 
\boldsymbol{k}
 }
\big\|
 W (\xi_1, \xi_2 )^{-1}
\Delta_{ \boldsymbol{k} } \sigma (x, \xi_1, \xi_2)
\big\|_{ L^2_{ul , \, x, \xi_1, \xi_2}(\R^{3n})}. 
\end{equation*}
\end{defn}

In terms of these notations, 
the theorem reads as follows.

\begin{thm}\label{main-thm} 
Let $W \in \calM (\R^{2n})$ and suppose 
the restriction of $W$ to $\Z^n \times \Z^n$ belongs 
to the class $\calB (\Z^n \times \Z^n)$. 
Let $r_j \in [1, 2]$, $j=1, \dots, n$. 
Then 
the bilinear 
pseudo-differential operator 
$T_{\sigma}$ is bounded from 
$L^2 (\R^n) \times L^2 (\R^n)$ 
to the amalgam space 
$(L^2, \ell^{r_1} \cdots \ell^{r_n})(\R^n)$ 
if 
$\sigma \in 
BS^{W}_{0,0} (\boldsymbol{s}; \R^n)$ 
with $\boldsymbol{s}=(s_{i,j})_{i,j}$ satisfying 
\begin{equation*}
s_{0,j} = 1/2, 
\quad 
s_{1,j}, s_{2,j }\ge 1/r_j -1/2, 
\quad 
s_{1,j} + s_{2,j} =1/r_j  
\end{equation*} 
for $j=1, \dots, n$. 
If in addition $r_1= \dots =r_n=r \in [1,2]$, 
then 
$T_{\sigma}$ is bounded 
from 
$L^2 (\R^n) \times L^2 (\R^n)$ 
to $L^r (\R^n)$ when $1<r\le 2$ or 
to $h^1 (\R^n)$ when $r=1$.  
\end{thm}

\begin{proof}
The assertion concerning the boundedness 
to $L^r$ or to $h^1$ 
directly follows from 
the assertion for the amalgam 
space with the aid of the embeddings  
$(L^2 , \ell^r) \hookrightarrow L^r$ 
for $1<r\le 2$ and 
$(L^2 , \ell^1) \hookrightarrow h^1$.

The boundedness to 
the amalgam space 
 follows from Proposition \ref{main-prop}. 
We decompose the symbol $\sigma$ 
by using the Littlewood--Paley partition: 
\begin{equation*}
\sigma (x,\xi_1,\xi_2)
=
\sum_{ \boldsymbol{k} \in (\N_0)^{3n}} 
\Delta_{ \boldsymbol{k} } 
\sigma (x, \xi_1, \xi_2)
\end{equation*}
Then the 
support of 
$\calF (\Delta_{\boldsymbol{k}} \sigma )$ 
is included in 
$\prod_{i=0}^2 ( \prod_{j=1}^n [-R_{i,j},R_{i,j}] )$ 
with 
$R_{i,j}= 2^{ k_{i,j} +1}$. 
Take $p_{1,j}, p_{2,j}$ such that 
$1/r_j - 1/2 \le 1/p_{1,j}, 1/p_{2,j} \le 1/2$ and 
$1/p_{1,j} + 1/p_{2,j} =1/r_j$ for $j=1, \dots, n$.  
Then  
Proposition \ref{main-prop} 
and the 
duality between amalgam spaces 
yield 
\begin{align*}
&
\|T_{ \Delta_{\boldsymbol{k}} \sigma }
\|_{L^2 \times L^2 \to (L^2, \ell^{r_1} \dots \ell^{r_n})}
\\
&\lesssim 
\bigg(
\prod_{j=1}^{n}
(2^{ k_{0,j} })^{1/2}
(2^{ k_{1,j} })^{ 1/p_{1,j} }
(2^{ k_{2,j} })^{ 1/p_{2,j} }
\bigg) 
\big\| W (\xi_1, \xi_2)^{-1}
 \Delta_{\boldsymbol{k}} \sigma 
(x, \xi_1, \xi_2) 
\big\|_{ L^2_{ul} }. 
\end{align*} 
Taking sum over $\boldsymbol{k} \in (\N_0)^{3n}$, 
we obtain  
\begin{align*}
&\|T_{ \sigma }
\|_{L^2 \times L^2 \to (L^2, \ell^{r_1} \dots \ell^{r_n})}
\le 
\sum_{ \boldsymbol{k} \in (\N_0)^{3n} }
\|T_{ \Delta_{\boldsymbol{k}} \sigma }
\|_{L^2 \times L^2 \to (L^2, \ell^{r_1} \dots \ell^{r_n})}
\\
&
\lesssim 
\sum_{ \boldsymbol{k} \in (\N_0)^{3n} }
\bigg(
\prod_{j=1}^{n}
(2^{ k_{0,j} })^{1/2}
(2^{ k_{1,j} })^{ 1/p_{1,j} }
(2^{ k_{2,j} })^{ 1/p_{2,j} }
\bigg) 
\big\| W (\xi_1, \xi_2)^{-1} 
 \Delta_{\boldsymbol{k}} 
\sigma (x, \xi_1, \xi_2) 
\big\|_{ L^2_{ul} }
\\
&=
\|\sigma\|_{ BS^{W}_{0,0} ( \boldsymbol{s}; \R^n) } 
\end{align*}
with $s_{0,j}=1/2$, $s_{1,j}=1/p_{1,j}$, and 
$s_{2,j}=1/p_{2,j}$, 
which is the desired result. 
\end{proof}

\subsection{Another theorem for symbols with 
limited smoothness}
\label{subsectionvariant}

In this subsection, 
we give a variant of Theorem \ref{main-thm}. 
Here to measure the smoothness 
of symbols, we use different Besov type norms 
which are defined below. 
It is easy to see that 
these Besov type norms also 
do not depend, up to the equivalence of norms, 
on the choice of the 
Littlewood--Paley partition of unity 
involved in the definition.

\begin{defn}\label{defBSW2} 
Let $W\in \calM (\R^{2n})$. 
Let $\{\psi^{(n)}_{k}\}_{ k\in \N_0 }$ 
be a Littlewood--Paley 
partition of unity on $\R^n$ 
and write 
\begin{equation*}
\Delta^{\ast}_{\boldsymbol{k}} 
\sigma (x, \xi_1, \xi_2)
=
\psi^{(n)}_{k_0} ( D_x )
\psi^{(n)}_{k_1} ( D_{\xi_1} )
\psi^{(n)}_{k_2} ( D_{\xi_2} )
\sigma (x, \xi_1, \xi_2) 
\end{equation*}
for $\boldsymbol{k}=(k_0, k_1, k_2)\in (\N_0)^3$. 
For $s_0, s_1, s_2 \in [0, \infty)$, 
we denote by $BS^{W, \ast}_{0,0} (s_0, s_1, s_2; \R^n) $ 
the set of all $\sigma \in L^{\infty} ((\R^n)^3)$ for which 
the following norm is finite: 
\begin{align*}
&\|\sigma\|_{ BS^{W, \ast}_{0,0} (s_0, s_1, s_2; \R^n) } 
\\
&=
 \sum_{k_0, k_1, k_2 \in \N_0}\, 
 2^{ s_0 k_0 + s_1 k_1 + s_2 k_2 }
\big\|
W (\xi_1, \xi_2 )^{-1}
\Delta^{\ast}_{\boldsymbol{k}} 
\sigma (x, \xi_1, \xi_2)
\big\|_{ L^2_{ul} (\R^{3n}) }. 
\end{align*}
\end{defn}

The following theorem can be deduced from 
Proposition \ref{main-prop} just in the same way 
as in Proof of Theorem \ref{main-thm}. 
We omit the proof.

\begin{thm}\label{main-thm-2} 
Let 
$W \in \calM (\R^{2n})$ 
and suppose 
the restriction of $W$ to $\Z^n \times \Z^n$ belongs 
to the class $\calB (\Z^n \times \Z^n)$. 
Let $r \in [1,2]$. 
Then the bilinear 
pseudo-differential operator 
$T_{\sigma}$ is bounded from 
$L^2 (\R^n) \times L^2 (\R^n)$ 
to the amalgam space 
$(L^2, \ell^r)(\R^n)$ 
if 
$\sigma \in 
BS^{W, \ast}_{0,0} (s_0, s_1, s_2 ; 
\R^n)$ 
with  $s_0=n/2$, 
$s_1, s_2 \ge n/r -n/2$, and 
$s_1 + s_2 =n/r$. 
In particular, 
under the same assumptions, 
$T_{\sigma}$ is bounded 
from 
$L^2 (\R^n) \times L^2 (\R^n)$ 
to 
$L^r (\R^n)$ in the case $1<r\le 2$ 
or to $h^1 (\R^n)$ in the case 
$r=1$. 
\end{thm}

\begin{rem}\label{remark-on-main-theorems}
We should compare 
Theorems \ref{main-thm} and \ref{main-thm-2}.  
In fact, the assertion of 
Theorem \ref{main-thm-2} for the case $1\le r<2$ 
is covered by Theorem \ref{main-thm}.  
To see this, we denote 
by 
$BS^{W}_{0,0}( (t_0)^n, (t_1)^n, (t_2)^n; \R^n )$ 
the class $BS^{W}_{0,0}( \boldsymbol{s}; \R^n )$ 
with $\boldsymbol{s}=(s_{i,j})$ given by 
\begin{equation*}
s_{0,j}=t_0, 
\;\; 
s_{1,j}=t_1, 
\;\; 
s_{2,j}=t_2, 
\;\; 
j=1, \dots, n.  
\end{equation*} 
With this special class of symbols, 
Theorem \ref{main-thm} for the case 
$r_1=\dots =r_n=r$ 
asserts that 
$T_{\sigma}$ is bounded from 
$L^2 (\R^n) \times L^2 (\R^n)$ to 
the amalgam space $(L^2, \ell^r)(\R^n)$ 
if $\sigma \in BS^{W}_{0,0}( (t_0)^n, (t_1)^n, (t_2)^n; \R^n )$ with 
$t_0=1/2$, $t_1, t_2 \ge 1/r-1/2$, and $1/t_1 + 1/t_2 =1/r$. 
This assertion is stronger than Theorem \ref{main-thm-2} 
in the case $1\le r<2$. 
This follows from the fact that the inclusion 
\begin{equation}\label{inclusion}
BS^{W, \ast}_{0,0} (n t_0, n t_1, n t_2; \R^n ) 
\hookrightarrow 
BS^{W}_{0,0}( (t_0)^n, (t_1)^n, (t_2)^n; \R^n ) 
\end{equation}
holds for $t_0, t_1, t_2 >0$. 
This inclusion, in a slightly different form,  
is already proved in \cite[Appendix A2 (i)]{boulkhemair 1995}. 
Here we give a brief proof for the reader's convenience. 
To prove 
\eqref{inclusion}, 
notice that 
\begin{equation}\label{DeltaDelta}
\Delta_{\boldsymbol{k}} \sigma (x, \xi_1, \xi_2)
=\sum_{ \boldsymbol{m} } 
\Delta_{\boldsymbol{k}} 
\Delta^{\ast}_{ \boldsymbol{m} } 
\sigma (x, \xi_1, \xi_2)   
\end{equation}
and that $\Delta_{\boldsymbol{k}} 
\Delta^{\ast}_{ \boldsymbol{m} } \neq 0$ only if 
\begin{equation}\label{restriction-mk}
\max \{k_{i,1}, \dots, k_{i, n}\} - c <m_i < \max 
\{k_{i,1}, \dots, k_{i, n}\} + c, 
\quad i=0,1,2,   
\end{equation}
where $c$ is a constant 
depending only on $n$. 
Using the property of $W \in \calM (\R^{2n})$ 
given in Proposition \ref{moderateproperties} (2), 
we see that the estimate 
\[
\|W(\xi_1, \xi_2)^{-1} 
\Delta_{\boldsymbol{k}} 
\tau (x, \xi_1, \xi_2)
\|_{L^2_{ul}}
\lesssim 
\|W(\xi_1, \xi_2)^{-1} 
\tau (x, \xi_1, \xi_2)
\|_{L^2_{ul}}  
\]
with an implicit constant 
independent of 
$\boldsymbol{k}$ 
holds for all bounded functions $\tau$ 
on $(\R^n)^3$.  
Thus from \eqref{DeltaDelta} we have 
\begin{equation}\label{DeltaDeltaL2ul}
\|W(\xi_1, \xi_2)^{-1} 
\Delta_{\boldsymbol{k}} \sigma (x, \xi_1, \xi_2)
\|_{L^2_{ul}} 
\lesssim 
\sum_{ 
\boldsymbol{m} : 
\eqref{restriction-mk} } 
\|W(\xi_1, \xi_2)^{-1} 
\Delta^{\ast}_{\boldsymbol{m}} 
\sigma (x, \xi_1, \xi_2)
\|_{L^2_{ul}}.  
\end{equation}
If $t_0, t_1, t_2>0$, then 
we have 
\[ 
\sum_{\boldsymbol{k}: \eqref{restriction-mk}} 
2^{
t_0 (k_{0,1} + \dots + k_{0,n}) 
+t_1 (k_{1,1} + \dots + k_{1,n}) 
+t_2 (k_{2,1} + \dots + k_{2,n}) 
}
\approx 
2^{n t_0 m_0 + nt_1 m_1 + nt_2 m_2}.   
\] 
Hence, from 
\eqref{DeltaDeltaL2ul},  
we obtain 
\[
\|\sigma\|_{ BS^{W}_{0,0} 
((t_0)^n, (t_1)^n, (t_2)^n; \R^n) } 
\lesssim 
\|\sigma\|_{ 
BS^{W, \ast}_{0,0} 
( n t_0, n t_1, n t_2; \R^n) }  
\]
as desired. 
\end{rem}

\subsection{Symbols with classical derivatives} 
\label{subsectionClassical}

In this subsection, we show 
that symbols that have classical derivatives 
up to certain order satisfy the conditions of 
Theorems \ref{main-thm} and \ref{main-thm-2}.

\begin{prop}
\label{classicalderivative} 
Let $\sigma = \sigma (x, \xi_1, \xi_2)$ 
be a bounded measurable 
function on $(\R^n)^3$ 
and $W \in \calM ( \R^{2n} )$. 
\begin{enumerate}
\setlength{\itemindent}{0pt} 
\setlength{\itemsep}{3pt} 
\item 
Let $\boldsymbol{s}=(s_{i,j})\in [0, \infty)^{3n}$. 
Suppose 
\[
 |\partial_{x_1}^{ \alpha_{0,1} }
\dots 
\partial_{x_n}^{ \alpha_{0,n} }
\partial_{ \xi_{1,1} }^{ \alpha_{1,1} }
\dots 
\partial_{ \xi_{1,n} }^{ \alpha_{1,n} }
\partial_{ \xi_{2,1} }^{ \alpha_{2,1} }
\dots 
\partial_{ \xi_{2,n} }^{ \alpha_{2,n} }
\sigma (x, \xi_1, \xi_2)
|
\le 
W (\xi_1, \xi_2) 
\]
for 
$\alpha_{i,j} \le [s_{i,j}]+1$. 
Then 
$\sigma \in 
BS^{W}_{0,0} ( \boldsymbol{s}; \R^n)$.  
\item 
Let $s_0, s_1, s_2\in [0, \infty)$. 
Suppose 
\[
 |
\partial_{x}^{ \alpha_{0} }
\partial_{ \xi_{1} }^{ \alpha_{1} }
\partial_{ \xi_{2} }^{ \alpha_{2} }
\sigma (x, \xi_1, \xi_2)
|
\le 
W (\xi_1, \xi_2) 
\]
for 
$\alpha_{i}\in (\N_0)^n$ 
with 
$|\alpha_{i}| \le [s_{i}]+1$. 
Then 
$\sigma \in 
BS^{W, \ast}_{0,0} ( s_0, s_1, s_2 ; \R^n)$. 
\end{enumerate}
To be precise, 
the above assumptions should be understood that 
the derivatives of $\sigma$ taken 
in the sense of distribution 
are functions in $L^{\infty}(\R^{3n})$ 
and they are bounded by $W(\xi_1, \xi_2)$ 
almost everywhere.  
\end{prop}

\begin{proof} 
It is sufficient to treat 
$\sigma$ of class $C^{\infty}$. 
In fact, by using appropriate mollifier 
we can derive the result for general $\sigma$ 
from the result for 
$\sigma$ of class $C^{\infty}$. 
Since the claims (1) and (2) can be proved 
in almost the same way, 
here we shall give a proof of (2) and leave 
the proof of 
(1) to the reader.

Suppose $\sigma$ is $C^{\infty}$ and 
satisfies the assumption of (2). 
We write 
$N_i= [s_i]+1$ and $\psi = \psi^{(n)} $.

First consider 
$\Delta^{\ast}_{\boldsymbol{k}} \sigma $ for 
$k_0,k_1,k_2 \geq 1$. 
Recall that 
$\psi_{k}(\xi)= \psi (\xi/ 2^k)$ for $k\ge 1$ 
and $\psi \in C_{0}^{\infty} (\R^n)$ 
satisfies  
$\supp \psi \subset \{ 1/2 \le | \xi | \le 2\}$. 
The inverse Fourier transform 
$ \check\psi $ satisfies the 
moment condition 
$\int x^\alpha \check \psi (x)\, dx
= i^{|\alpha|} \partial^\alpha \psi (0)=0$. 
Thus, using the Taylor expansion with respect to 
the third variable of the symbol,
we have
\begin{align*}
&
\Delta^{\ast}_{\boldsymbol{k}} 
\sigma (x,\xi_1,\xi_2)
\\
&=
2^{n(k_0+k_1+k_2)} \int_{(\R^{n})^3} 
\check \psi (2^{k_0}y) 
\check \psi (2^{k_1}\eta_1) 
\check \psi (2^{k_2}\eta_2)
\\
&\quad\times
\Bigg\{ \sigma (x-y, \xi_1-\eta_1, \xi_2-\eta_2) 
- 
\sum_{|\alpha_2| < N_2} 
\frac{(-\eta_2)^{\alpha_2}}{\alpha_2!} 
\big( \partial_{\xi_2}^{\alpha_2}\sigma \big)
 (x-y,\xi_1-\eta_1,\xi_2)
\Bigg\} \, dY
\\
&=
2^{n(k_0+k_1+k_2)} \int_{(\R^{n})^3} 
\check \psi (2^{k_0}y) 
\check \psi (2^{k_1}\eta_1) 
\check \psi (2^{k_2}\eta_2)
\\
&\quad\times 
\sum_{|\alpha_2| = N_2} 
\frac{(-\eta_2)^{\alpha_2}}{\alpha_2!} 
\int_0^1 
N_2(1-t_2)^{N_2-1} 
\big( \partial_{\xi_2}^{\alpha_2}\sigma \big) 
(x-y,\xi_1-\eta_1,\xi_2-t_2 \eta_2) 
\, dt_2 dY,
\end{align*}
where $dY = dy d\eta_1 d\eta_2$.  
Repeating the same argument to the variables 
$\eta_1$ and $y$, we obtain
\begin{align}
\label{taylorofsymbol}
\begin{split}
&
\Delta^{\ast}_{\boldsymbol{k}}
\sigma (x,\xi_1,\xi_2)
\\
&=
2^{n(k_0+k_1+k_2)} 
\sum_{|\alpha_0| = N_0} \frac{1}{\alpha_0 !} 
\sum_{|\alpha_1| = N_1} \frac{1}{\alpha_1 !} 
\sum_{|\alpha_2| = N_2} \frac{1}{\alpha_2 !} 
\\
&\quad\times
\int_{(\R^{n})^3}
\check \psi (2^{k_0}y) (-y)^{ \alpha_0 } 
\check \psi (2^{k_1}\eta_1) (-\eta_1)^{\alpha_1} 
\check \psi (2^{k_2}\eta_2) (-\eta_2)^{\alpha_2} 
\\
&\quad\times
\int_{[0,1]^3} 
\bigg( \prod_{i=0}^2 
N_i (1-t_i)^{N_i-1} \bigg)
\big( \partial_{x}^{\alpha_0 } 
\partial_{\xi_1}^{\alpha_1} 
\partial_{\xi_2}^{\alpha_2} 
\sigma \big)
(x-t_0y, \xi_1-t_1 \eta_1, \xi_2-t_2 \eta_2 ) 
\, dTdY,
\end{split}
\end{align}
where $dT = dt_0 dt_1 dt_2$. 
If $\sigma$ satisfies the assumption of (2), then 
for $\alpha_0, \alpha_1, \alpha_2$ 
with $|\alpha_i|=N_i$ 
we have 
\begin{align*}
\left|
\big( \partial_{x}^{\alpha_0 } 
\partial_{\xi_1}^{\alpha_1} 
\partial_{\xi_2}^{\alpha_2} 
\sigma \big)
(x-t_0y, \xi_1-t_1 \eta_1, \xi_2-t_2 \eta_2 ) 
\right|
&\le 
W (\xi_1-t_1 \eta_1, \xi_2-t_2 \eta_2) 
\\
&\lesssim
W (\xi_1, \xi_2) 
\langle \eta_1 \rangle^{L}
\langle \eta_2 \rangle^{L},
\end{align*}
where the latter inequality 
follows from the assumption  
$W \in \calM (\R^{2n})$ 
and 
$L$ is a constant depending 
on $W$ 
(see Proposition \ref{moderateproperties} (2)). 
Hence 
\begin{align*}
&
| \Delta^{\ast}_{\boldsymbol{k}}
\sigma (x,\xi_1,\xi_2) |
\\
&\lesssim
2^{n(k_0+k_1+k_2)} 
W (\xi_1, \xi_2 ) 
\\
&\quad\times 
\int_{(\R^{n})^3}
\left| \check \psi (2^{k_0} y) \right| \, 
|y|^{N_0} 
\left| \check \psi (2^{k_1} \eta_1) \right| \, 
|\eta_1|^{N_1} 
\langle \eta_1 \rangle^{L}
\left| \check \psi (2^{k_2} \eta_2) \right| \, 
|\eta_2|^{N_2} 
\langle \eta_2 \rangle^{L}
\, dY
\\
&\lesssim 
\,2^{-k_0 N_0} 
\,2^{-k_1 N_1} 
\,2^{-k_2 N_2} 
W (\xi_1 , \xi_2 ) 
\end{align*}
for all $x,\xi_1,\xi_2 \in \R^n$ and 
all $k_0,k_1,k_2 \geq 1$.

If one of $k_i$ is zero, then 
by avoiding usage of the moment condition 
and the Taylor expansion for the corresponding variables,
we also obtain the same conclusion as above.

Thus we have 
\begin{equation*}
\| 
W (\xi_1, \xi_2 )^{-1}
\Delta^{\ast}_{\boldsymbol{k}}\sigma (x,\xi_1,\xi_2) 
\|_{L^{2}_{ul}}
\lesssim
\,2^{-k_0 N_0 } 
\,2^{-k_1 N_1 } 
\,2^{-k_2 N_2 }
\end{equation*}
for all $k_0,k_1,k_2 \in\N_0$. 
Since $N_i = [s_i]+1 > s_i$, 
the above inequalities imply  
\[
\sum_{k_0, k_1, k_2 \in \N_0}
2^{s_0 k_0 + s_1 k_1 + s_2 k_2} 
\| W (\xi_1, \xi_2)^{-1} 
\Delta^{\ast}_{\boldsymbol{k}} 
\sigma (x,\xi_1,\xi_2) 
\|_{L^{2}_{ul}} 
\lesssim 1.
\] 
This completes the proof.
\end{proof}

\subsection{Proof of Theorem \ref{main-thm-1}}
\label{subsection-Proof-Main-1}

Here we give a proof of
Theorem \ref{main-thm-1}.

\begin{proof} 
We prove the assertion (2) first. 
Suppose $V \in \calB (\Z^n \times \Z^n)$ 
and $\sigma \in 
BS^{ \widetilde{V} }_{ 0,0 } (\R^n)$. 
We take a function $V^{\ast}$ as mentioned in 
Proposition \ref{Vast}. 
By Proposition \ref{moderateproperties} (2), 
it follows that $\widetilde{V} \lesssim V^{\ast}$ 
and hence $\sigma \in BS^{ V^{\ast} }_{ 0,0 } (\R^n)$. 
Proposition \ref{classicalderivative} 
implies that 
$\sigma$ also satisfies  
the assumptions of Theorems \ref{main-thm} 
and \ref{main-thm-2} 
with $W = V^{\ast}$ and $r_1= \dots =r_n=r=1$, 
and the boundedness of $T_{\sigma}$ 
follows.

Next, we shall prove the assertion (1). 
The basic idea of this part of proof 
goes back to 
\cite[Proof of Lemma 6.3]{MT-2013},

Let $V$ be a nonnegative bounded function on 
$\Z^n \times \Z^n$ and $0<r<\infty$. 
We assume 
$\mathrm{Op} (BS^{\widetilde{V}}_{0,0} ) 
\subset 
B (L^2 \times L^2 \to L^r )$ 
with 
$\widetilde{V}$ defined as in 
Theorem \ref{main-thm-1}. 
By the closed graph theorem, 
it follows that 
there exist a positive integer $M$
and a positive constant $C$ such that
\begin{equation}\label{inequality001}
\|T_{\sigma}\|_{L^2 \times L^2 \to L^r}
\le C\max_{|\alpha|, |\beta_1|, |\beta_2| \le M}
\left\| \widetilde{V}(\xi_1, \xi_2)^{-1} 
\partial^{\alpha}_x\partial^{\beta_1}_{\xi_1}
\partial^{\beta_2}_{\xi_2}\sigma(x,\xi_1,\xi_2)\right\|_{L^{\infty}}
\end{equation}
for all bounded smooth functions $\sigma$ on $(\R^n)^3$ 
(see \cite[Lemma 2.6]{BBMNT}). 
Our purpose is to prove the inequality 
\eqref{BL222}. 
For this,  
it is sufficient to consider 
$A, B, C \in \ell^2 (\Z^n)$ such that 
$A(\mu)=B(\mu) = C(\mu)=0$ except for 
a finite number of $\mu \in \Z^n$.

Take $\varphi,\widetilde{\varphi} \in \calS(\R^n)$ 
such that 
\begin{align}
&\mathrm{supp}\, \widetilde{\varphi} 
\subset [-1/2,1/2]^n, 
\quad
\widetilde{\varphi}=1 \ \text{on $[-1/4,1/4]^n$},  
\quad 
\mathrm{supp}\, \varphi \subset [-1/4,1/4]^n, 
\nonumber
\\
&
|\calF ^{-1} \varphi | \ge 1 \;\; 
\text{on $[-\pi,\pi ]^n$}. 
\label{varphi-live}
\end{align}
Take a sequence of real numbers 
$\{\epsilon_k\}_{k \in \Z^n}$
such that $\sup_{k \in \Z^n} |\epsilon_k| \le 1$, 
and set
\begin{equation*}
\sigma (\xi_1,\xi_2)
=\sum_{ k_1,k_2 \in \Z^n }
\epsilon_{ k_1+k_2 } V(k_1, k_2) 
\widetilde{\varphi}(\xi_1-k_1)
\widetilde{\varphi}(\xi_2-k_2). 
\end{equation*}
Then we have 
\begin{equation}\label{estimate-sigma}
|\partial^{\beta_1}_{\xi_1}\partial^{\beta_2}_{\xi_2}
\sigma (\xi_1,\xi_2)|
\le C_{\beta_1,\beta_2} 
\widetilde{V} (\xi_1, \xi_2) 
\end{equation}
with $C_{\beta_1,\beta_2}$ independent of 
the sequence $\{\epsilon_k\}$. 
We define $f_1, f_2 \in \calS (\R^n)$ by 
\begin{align*}
&
\widehat{f_{1}}(\xi_1 )
=\sum_{\nu_1 \in \Z^n}
B(\nu_1)
\varphi(\xi_1 -\nu_1), 
\\
&
\widehat{f_{2}}(\xi_2 )
=\sum_{\nu_2 \in \Z^n}
C(\nu_2)
\varphi(\xi_2 -\nu_2). 
\end{align*}
Then 
$
f_{1} (x)
=
\sum_{\nu_1 \in \Z^n}
B(\nu_1)
e^{i \nu_1 \cdot x} 
\calF^{-1}\varphi (x)$ 
and hence, 
using Parseval's identity and 
\eqref{varphi-live}, 
we have 
$\|f_1\|_{L^2} \approx \|B\|_{\ell^2}$. 
Similarly 
$\|f_2\|_{L^2} \approx \|C\|_{\ell^2}$. 
From the situation of 
the supports of 
$\varphi$ and 
$\widetilde{\varphi}$, 
we have 
\begin{align*}
T_{\sigma}(f_1,f_2)(x)
&=
\sum_{ \nu_1, \nu_2 \in \Z^n }
\epsilon_{\nu_1+\nu_2} 
V(\nu_1, \nu_2) 
B(\nu_1) C(\nu_2) 
e^{i (\nu_1 + \nu_2)\cdot x} 
\calF^{-1} \varphi (x)^2
\\
&
=\sum_{k}
\epsilon_k d_k e^{i k \cdot x} 
\calF^{-1} \varphi (x)^2,
\end{align*}
where
\begin{equation}\label{def-dk}
d_k = 
\sum_{\nu_1 + \nu_2 =k} 
V(\nu_1, \nu_2) 
B(\nu_1) C(\nu_2). 
\end{equation}
Notice that 
$d_k\neq 0$ only for a finite number of $k$'s  
by virtue of our assumptions on $B$ and $C$.

Now from \eqref{inequality001}, \eqref{estimate-sigma}, 
and from the estimates 
of the $L^2$ norms of $f_1$ and $f_2$ mentioned 
above, we have 
\begin{equation*}
\|T_{\sigma}(f_1,f_2)\|_{L^r}
\lesssim \|f_1\|_{L^2}\|f_2\|_{L^2}
\approx \|B\|_{\ell^2} 
\|C\|_{\ell^2}.
\end{equation*}
By \eqref{varphi-live}, we have
\[
\|T_{\sigma} (f_1, f_2)\|_{L^r}^r
\ge 
\int_{[-\pi, \pi]^n}
\bigg| \sum_{ k }
\epsilon_k 
d_k e^{i k \cdot x} \bigg|^r dx. 
\]
Hence 
\begin{equation}\label{inequality002}
\int_{[-\pi, \pi]^n}
\bigg| \sum_{ k }
\epsilon_k 
d_k e^{i k \cdot x} \bigg|^r dx
\lesssim \left(
\|B\|_{\ell^2} 
\|C\|_{\ell^2}
\right)^r.
\end{equation}
It should be noticed that
the implicit constant in 
\eqref{inequality002} does  
not depend on $\{\epsilon_k\}$.

We choose 
$\epsilon_k = \epsilon_k (\omega)$  
to be identically distributed 
independent random variables on a probability 
space, each of which  
takes $+1$ and $-1$ with 
probability $1/2$. 
Then integrating 
over $\omega$ and using 
Khintchine's inequality, 
we have 
\begin{equation}\label{Khintchine}
\int 
\Big(\text{the left hand side of 
\eqref{inequality002}} 
\Big)\, dP (\omega)
\approx 
\bigg(\sum_{k}
|d_k|^2 \bigg)^{r/2} 
\end{equation}
(for Khintchine's inequality, 
see, e.g., \cite[Appendix C]{grafakos 2014c}).

Combining 
\eqref{def-dk}, \eqref{inequality002}, and 
\eqref{Khintchine}, 
we obtain 
\[
\bigg\|
\sum_{\nu_1 + \nu_2 =k} 
V(\nu_1, \nu_2) 
B(\nu_1) C(\nu_2)
\bigg\|_{\ell^2_k}
\lesssim 
\|B\|_{\ell^2}
\|C\|_{\ell^2}, 
\]
which is equivalent to \eqref{BL222}. 
This completes the proof of 
Theorem \ref{main-thm-1}. 
\end{proof}

\subsection{A theorem of 
Grafakos--He--Slav\'ikov\'a with 
some generalization}
\label{subsectionGHS} 

The theorem given below 
is a generalization of the theorem of 
Grafakos--He--Slav\'ikov\'a \cite{GHS}. 
We shall prove this theorem by using 
Theorem \ref{main-thm-1}.

\begin{thm}\label{GHSq<4}
Suppose $\sigma \in BS^{0}_{0,0}(\R^n)$ 
with the notation of \eqref{defBSm00} 
and suppose the function 
$V(\xi_1, \xi_2)=
\sup_{x\in \R^n}|\sigma (x, \xi_1, \xi_2)|$ 
belongs to $L^{q}_{\xi_1, \xi_2} (\R^{2n})$
for some $0<q <4$. 
Then the bilinear pseudo-differential operator 
$T_{\sigma}$ is bounded from $L^2 \times L^2 $ 
to the amalgam space $(L^2, \ell^1)$. 
In particular, $T_{\sigma}$ is 
bounded from $L^2 \times L^2 $ to 
$h^{1} \cap L^2 $. 
\end{thm}

\begin{proof} 
We assume $V \in L^q (\R^{2n})$ with $1\le q<4$. 
The assumption $q\ge 1$ gives no additional restriction 
since $\sigma$ already belongs to $L^{\infty}$ by 
the assumption $\sigma \in BS^{0}_{0,0}(\R^n)$. 
In the following argument, $N$ denotes a fixed 
sufficiently large positive number that 
depends only on the dimension $n$.

We take a Littlewood-Paley partition 
of unity $\{\psi_k\}$ on $\R^{3n}$ and 
decompose $\sigma$ as  
\[
\sigma (x, \xi_1, \xi_2) 
= 
\sum_{k=0}^{\infty} 
\psi_{k} ( D_{x, \xi_1, \xi_2} ) \sigma (x, \xi_1, \xi_2)
=
\sum_{k=0}^{\infty} 
\sigma_{k} (x, \xi_1, \xi_2). 
\]
In order to show  
$T_{\sigma}: L^2 \times L^2 \to (L^2, \ell^1)$, 
we shall prove 
\begin{equation}\label{absolutesum}
\sum_{k=0}^{\infty} 
\|T_{\sigma_k}\|_{L^2 \times L^2 \to (L^2, \ell^1)} 
< \infty. 
\end{equation}

We define $V_k$ by 
\begin{align*}
V_{k}(\xi_1, \xi_2) 
= \int_{\R^{2n}} 
V(\eta_1, \eta_2 ) 
2^{ 2 k n } (1 + 2^k |\xi_1 - \eta_1| + 2^k |\xi_2 - \eta_2|) ^{-N}\, 
d\eta_1 d\eta_2. 
\end{align*}
We shall derive estimates of $\sigma_k$ in 
terms of $V_k$.

Firstly, 
\begin{equation}\label{dasigmakVk}
\big|
\partial_{x, \xi_1, \xi_2}^{\alpha} 
\sigma_{k} (x, \xi_1, \xi_2)
\big| 
\le C_{\alpha} 2^{k |\alpha|} 
V_{k} (\xi_1, \xi_2). 
\end{equation}
To see this, consider first the case $k\ge 1$. 
Then recall that 
the function $\psi_{k}$ is of the form 
$\psi_{k}=\psi (2^{-k} \cdot )$ 
with $\psi \in \calS (\R^{3n})$. 
Hence 
the derivative on the left hand side can be written as 
\begin{align*}
&
\partial_{x, \xi_1, \xi_2}^{\alpha} 
\sigma_{k} (x, \xi_1, \xi_2)
=\big(
(\partial^{\alpha} \calF^{-1}\psi_k ) \ast 
\sigma \big)
(x, \xi_1, \xi_2)
\\
&
= 
\int_{\R^{3n}} 
2^{3 k n}\,  2^{ k |\alpha|} 
(\partial^{\alpha} \calF^{-1}\psi ) 
(2^k (x-y, \xi_1-\eta_1, \xi_2 - \eta_2) ) 
\sigma (y, \eta_1, \eta_2)\, dy d\eta_1 d\eta_2.  
\end{align*}
Since $\psi \in \calS$ and since 
$\sigma$ is bounded by $V$, 
the integrand on the right hand side is 
bounded by 
\[
C_{\alpha} 
2^{3 k n}\,  2^{ k |\alpha|} 
(1+ 2^k |x-y| )^{-N} 
(1+ 2^k |\xi_1 -\eta_1| + 2^k |\xi_2 - \eta_2| ) ^{-N} 
V (\eta_1, \eta_2)
\]
and thus the estimate \eqref{dasigmakVk} follows. 
Proof for $k=0$ is similar.

Secondly, 
\begin{equation}\label{dasigmakLinfty}
\big|
\partial_{x, \xi_1, \xi_2}^{\alpha} 
\sigma_{k} (x, \xi_1, \xi_2)
\big| 
\le C_{\alpha, L} 2^{- k L}, 
\end{equation}
where $L\in \N$ can be taken arbitrarily large. 
For $k=0$, this estimate is obvious from the assumption 
$\sigma \in BS^{0}_{0,0}$. 
Suppose $k\ge 1$. 
We write $X=(x, \xi_1, \xi_2)$ and 
$Y=(y, \eta_1, \eta_2)$. 
Then, since $\psi_k (X)= \psi (2^{-k} X)$ and 
$\calF^{-1}\psi$ satisfies 
the moment condition 
$\int X^{\alpha} \calF^{-1}\psi (X)\, dX =0$, 
we have  
\begin{align*}
&
\partial_{X}^{\alpha} 
\sigma_{k} (X)
=
\big(
\calF^{-1}\psi_{k}
\ast 
 (\partial^{\alpha} \sigma) 
\big)
(X)
\\
&
= 
\int_{\R^{3n}} 
2^{3 k n} 
( \calF^{-1}\psi  ) 
(2^k Y )  
\bigg(
(\partial^{\alpha} \sigma) (X-Y) 
- 
\sum_{|\beta|<L} 
\frac{ \partial^{\beta + \alpha} \sigma (X) }{\beta !}\, 
(-Y)^{\beta}\bigg)
\, dY.   
\end{align*}
Since $\psi \in \calS$ and since 
the derivatives of $\sigma$ are bounded, 
the integrand on the right hand side is 
bounded by 
\[
C_{\alpha, L}
2^{3 k n} 
(1+ 2^k |Y| )^{-N-L} 
|Y|^{L}
=
C_{\alpha, L}
 2^{-kL}\, 2^{3 k n} 
(1+ 2^k |Y| )^{-N-L} 
|2^k Y|^{L}
\]
and thus the estimate \eqref{dasigmakLinfty} follows.

We consider the symbol  
\begin{equation*}
\widetilde{\sigma}_{k} (x, \xi_1, \xi_2)
=\sigma_{k} (2^k x, 2^{-k} \xi_1, 2^{-k} \xi_2). 
\end{equation*}
For bilinear pseudo-differential operators, 
a simple change of variables yields the formula 
\begin{equation*}
T_{\sigma_{k} } (f_1, f_2) (2^k x)
=
T_{\widetilde{\sigma}_{k} } 
(f_1 (2^k \cdot ), f_2 (2^k \cdot)) (x). 
\end{equation*}
For the norm of $(L^2, \ell^1)(\R^n)$, 
there exists a real number $a$ such that 
\begin{equation*}
\|g (\lambda \cdot )
\|_{ (L^2, \ell^1) (\R^n)} 
\lesssim \lambda^{a}  
\|g\|_{ (L^2, \ell^1) (\R^n)}
\;\; \text{for} \;\; 0<\lambda\le 1. 
\end{equation*}
(In fact, we can take $a=-n$ and this is the optimal 
number; however, 
the exact value of $a$ is not necessary 
for our argument.) 
For the $L^2$ norm, we have 
\begin{equation*}
\|g (\lambda \cdot )\|_{L^2 (\R^n)} 
=\lambda^{-n/2} \|g\|_{L^2 (\R^n)}, 
\quad 
\lambda >0.  
\end{equation*}
Combining 
these formulas, 
we see that 
\begin{equation}\label{TskL2L2h1}
\|T_{\sigma_k } \|_{L^2 \times L^{2} 
\to (L^2, \ell^1) } 
\lesssim 
2^{-k(n+a)} 
\|T_{\widetilde{\sigma}_k } 
\|_{L^2 \times L^{2} \to (L^2, \ell^1) } . 
\end{equation}
We shall estimate the operator norms 
of $T_{\widetilde{\sigma}_k } $ 
by using Theorem \ref{main-thm-1}.

From 
\eqref{dasigmakVk} and 
\eqref{dasigmakLinfty},  
we have  
\begin{equation*}
| \partial^{\alpha}_{x, \xi_1, \xi_2}
\widetilde{\sigma}_k (x, \xi_1, \xi_2) |
\le 
C_{\alpha, L}\, 
2^{2k |\alpha|} 
W_{k}(\xi_1, \xi_2) 
\end{equation*}
with 
\begin{equation*}
W_{k}(\xi_1, \xi_2) = 
\min  
\{2^{-kL},\, 
V_{k} (2^{-k} \xi_1, 2^{-k} \xi_2)
\}. 
\end{equation*}
From the definition of $V_k$, 
we easily see that 
\begin{equation}\label{Wkmoderate}
|\xi_1 - \xi_1'|\le 1 \;\;
\text{and}\;\; 
|\xi_2 - \xi_2'|\le 1 \;\;
\Rightarrow 
\;\; 
W_{k}(\xi_1, \xi_2)\approx 
W_{k}(\xi_1', \xi_2'), 
\end{equation} 
where the implicit constants in $\approx$ do not 
depend on $\xi_i, \xi_i'$, and $k$. 
We have 
\[
\|W_k\|_{ L^q } 
\le 
\|V_k (2^{-k}\xi_1, 2^{-k}\xi_2) 
\|_{L^q _{\xi_1, \xi_2} (\R^{2n})}
=
2^{2kn/q}
\|V_k\|_{L^q } 
\le 
c 
2^{2kn/q}
\|V\|_{L^q} 
\]
with $c$ independent of $k$. 
Also 
$\|W_k\|_{ L^{\infty} } \le 2^{-kL}$.  
Thus, since $q<4<\infty$, we have 
\begin{equation}\label{WkL4}
\|W_k \|_{L^4 } 
\le 
\|W_k \|_{L^q }^{1-\theta} 
\|W_k \|_{L^{\infty} } ^{\theta}
\le 
c  
(2^{-kL})^{\theta}
(2^{2kn/q})^{1-\theta}
\|V\|_{L^q}^{1-\theta}, 
\end{equation}
where $1-\theta= q/4 $. 
From \eqref{Wkmoderate}, \eqref{WkL4},  
and Proposition \ref{weightL4weak}, 
we see that $W_k|(\Z^n\times \Z^n)$, 
the restriction of $W_k$ to $\Z^n\times \Z^n$, 
belongs to $\calB (\Z^n \times \Z^n)$ and 
the constant $c$ of \eqref{BL222} for 
$V=W_k|(\Z^n\times \Z^n)$ is bounded by 
a constant times \eqref{WkL4}. 
Hence, using Theorem \ref{main-thm-1}, we obtain 
\begin{equation}\label{TsktildeL2L2h1}
\|T_{\widetilde{\sigma}_k } 
\|_{L^2 \times L^{2} \to (L^2, \ell^1) }
\le 
C_{L}\, 2^{2kM}\, 
(2^{-kL})^{\theta}
(2^{2kn/q})^{1-\theta}
\|V\|_{L^q}^{1-\theta}, 
\end{equation}
where $M$ is a constant depending only on 
the dimension $n$.  
(Notice that, with the aid of the closed graph theorem, 
Theorem \ref{main-thm-1} actually 
gives an estimate of the operator norm 
of a pseudo-differential operator in 
terms of the norms of 
certain finite number of 
the derivatives of the symbol.) 
Since $L$ can be taken arbitrarily large, 
\eqref{TskL2L2h1} 
and 
\eqref{TsktildeL2L2h1} 
imply \eqref{absolutesum}. 
\end{proof}

\section{ Sharpness of the theorems }
\label{sectionsharpness}

In this section, 
we shall prove that our main theorems, 
Theorems \ref{main-thm-1}, 
\ref{main-thm},  
and \ref{main-thm-2}, 
are sharp in several senses. 
Here we consider the cases of the following special weights:  
\begin{align*}
&W_{m}(\xi_1, \xi_2)
= \langle (\xi_1, \xi_2) \rangle ^{m}, 
\quad  
m\in (-\infty, 0], 
\\
&
W_{m_1, m_2}(\xi_1, \xi_2)
= \langle \xi_1 \rangle ^{m_1} 
\langle \xi_2  \rangle ^{m_2}, 
\quad 
m_1, m_2 \in (-\infty, 0].  
\end{align*}
We denote the class 
$BS ^{W}_{0,0} (\R^n)$ of Definition \ref{BSW00} 
for 
$W=W_{m}$ and $W=W_{m_1, m_2}$ 
simply by  
$BS ^{m}_{0,0} (\R^n)$ and 
$BS ^{(m_1, m_2)}_{0,0} (\R^n)$, 
respectively. 
Thus the class $BS ^{m}_{0,0} (\R^n)$ is the same as 
the one defined by \eqref{defBSm00}.

\subsection{Sharpness of the order $-n/2$ }

We have already observed that   
$W_{m}$ with $m=-n/2$ 
and 
$W_{m_1, m_2}$  
with  
$m_1, m_2 <0$ and $m_1+ m_2 =-n/2$ 
belong to $\calB (\Z^n \times \Z^n)$
(see Example \ref{example-of-V} and 
the proof given in Section \ref{sectionWeight}). 
Here we shall see that these 
are critical weights among the weights 
$W_m$ and $W_{m_1, m_2}$. 
Firstly, 
the weight $W_{m}$ with $-n/2 < m\le 0$ does not 
belong to $\calB (\Z^n \times \Z^n)$  
as we have already observed in Proof of 
Proposition \ref{propertiesB} (3). 
Next, the weight 
$W_{m_1, m_2}$ with $m_1, m_2 \in (-\infty, 0]$     
does not belong to  $\calB (\Z^n \times \Z^n)$ 
if $m_1+ m_2 >-n/2$ 
or if $m_1+ m_2 =-n/2$ and 
$m_1 m_2=0$. 
To show this, observe that 
$W_{m_1+ m_2}(\xi_1, \xi_2) 
\le 
W_{m_1, m_2}(\xi_1, \xi_2)$. 
Thus if 
$W_{m_1, m_2} \in \calB (\Z^n \times \Z^n)$ 
then 
$W_{m_1+ m_2} \in \calB (\Z^n \times \Z^n)$, 
which is possible only when 
$m_1+ m_2 \le -n/2$.  
Also Proposition \ref{single-L2} implies that 
the functions $W_{0, -n/2}$ and 
$W_{-n/2, 0}$ do not belong to 
$\calB (\Z^n \times \Z^n)$.

\subsection{Sharpness of $r\in [1,2]$ }

The next proposition shows that 
the range $1\le r \le 2$ 
in Theorems \ref{main-thm-1}, \ref{main-thm}, 
and \ref{main-thm-2} is in a sense optimal.

\begin{prop}\label{sharp-range}
Let $0<r<\infty$, $m \in (-\infty, 0]$, 
and assume 
$ \mathrm{Op} (BS^{m}_{0,0} (\R^n)) 
\subset 
B(L^2 \times L^2 \to L^r )$. 
Then $r \ge 1$.  
Moreover, $r \le 2$ in the case $m=-n/2$.
\end{prop}

\begin{proof}
If the symbol $\sigma (x, \xi_1, \xi_2)$ is 
independent of $x$, then 
$\sigma$ is called a Fourier multiplier and 
$T_{\sigma}$ is called a bilinear Fourier multiplier operator. 
For bilinear Fourier multiplier operators, 
the following is known: 
if a nonzero Fourier multiplier operator $T_{\sigma}$ 
is bounded from $L^p \times L^q$ to $L^r$, 
$1 \le p,q <\infty$, and $0<r<\infty$,
then $1/p+1/q \ge 1/r$
(see \cite[Proposition 5]{GT} and 
\cite[Proposition 7.3.7]{grafakos 2014m}). 
Let $\sigma(\xi_1,\xi_2)$ be a 
nonzero function in $\calS((\R^n)^2)$. 
Then, since $\sigma(\xi_1,\xi_2)$ belongs to
$BS^{\widetilde{m}}_{0,0}$ for 
any $\widetilde{m} \le 0$,
the assumption of the proposition implies 
$T_{\sigma}: L^2 \times L^2 \to L^r$.
Hence, by the fact mentioned above, 
we must have $1/2+1/2 \ge 1/r$, that is, $r \ge 1$.

Next we show that $r\le 2$ in the case $m=-n/2$. 
Assume that
$T_{\sigma} : L^2 \times L^2 \to L^r$
for all $\sigma \in BS^{-n/2}_{0,0}$.
Let
$\Psi \in \calS((\R^n)^2)$
and $\psi \in \calS(\R^n)$ be such that
$\Psi (\zeta)=1$ on $ \{ 2^{-1/4} \le |\zeta| \le 2^{1/4} \} $,
$\mathrm{supp}\, \Psi \subset
\{2^{-1/2} \le |\zeta| \le 2^{1/2}\}$, 
$\mathrm{supp}\, \psi 
\subset \{2^{-3/4} \le |\eta| \le 2^{-1/4}\}$, 
and $\psi \neq 0$. 
We set
\begin{align*}
&\sigma(\xi_1, \xi_2)
=\sum_{j \in \N_0}
2^{-jn/2}\Psi ( 2^{-j} (\xi_1, \xi_2)),
\quad
(\xi_1, \xi_2) \in \R^{2n}, 
\\
&
\widehat{f_{1,k}}(\eta )= 
\widehat{f_{2,k}}(\eta )
=2^{-kn/2}\psi(2^{-k}\eta),
\quad  \eta \in \R^n, \quad k \in \N_0. 
\end{align*}
Then $\sigma \in BS^{-n/2}_{0,0}$
(in fact, $\sigma \in BS^{-n/2}_{1,0}$) 
and $\|f_{i,k}\|_{L^2}=\|\psi\|_{L^2}$ does not 
depend on $k$.
From the support conditions on 
$\Psi$ and $\psi$, 
we see that
$\Psi ( 2^{-j} (\xi_1, \xi_2)) 
\widehat{f_{1,k}}(\xi_1)\widehat{f_{2,k}}(\xi_2)$
equals
$\widehat{f_{1,k}}(\xi_1)\widehat{f_{2,k}}(\xi_2)$ if $j=k$
and vanishes if $j \neq k$.
Thus 
\[
T_{\sigma}(f_{1,k},f_{2,k})(x)
=2^{-kn/2}\left(2^{kn/2}\check{\psi}(2^k x)\right)^2
=2^{kn/2}\check{\psi}(2^{k}x)^2. 
\]
Hence our assumption implies that
\[
2^{kn(1/2-1/r)}
\approx \|T_{\sigma}(f_{1,k},f_{2,k})\|_{L^r}
\lesssim \|f_{1,k}\|_{L^2}\|f_{2,k}\|_{L^2} \approx 1,
\quad k \in \N_0, 
\]
which is possible only when $1/2-1/r \le 0$,
namely $r \le 2$.
\end{proof}

\subsection{Sharpness of $s_0, s_1, s_2$ in 
Theorem \ref{main-thm-2}}

In this subsection, we shall prove that the 
conditions on $s_0, s_1, s_2$ in 
Theorem \ref{main-thm-2} are sharp. 
First we shall prove the following.

\begin{prop}\label{sharp-smoothness}
Let $1 \le r \le 2$
and $\boldsymbol{s}
=(s_0,s_1,s_2) \in [0,\infty)^3$.
If all bilinear pseudo-differential operators 
$T_{\sigma}$ 
with symbols $\sigma$ on $(\R^n)^3$ satisfying
\begin{equation}\label{sharp-smoothness-assump}
\sup_{\boldsymbol{k} \in (\N_0)^3}
2^{\boldsymbol{k}\cdot\boldsymbol{s}}
\left\|\langle (\xi_1,\xi_2) \rangle^{n/2}
\Delta^{\ast}_{\boldsymbol{k}}\sigma(x,\xi_1,\xi_2)\right\|_{L^{\infty}_{x, \xi_1, \xi_2}((\R^n)^3)}
<\infty
\end{equation}
are bounded from $L^2 \times L^2$
to $L^r$, then $s_0 \ge n/2$,
$s_1,s_2 \ge n/r-n/2$, 
and $s_1+s_2 > n/r$.
\end{prop}

\begin{proof}
In this proof, we use nonnegative functions
$\varphi,  \theta \in \calS(\R^n)$
such that $\varphi (x)=1$ on $\{ |x| \le 1\}$,
$\mathrm{supp}\, \varphi \subset \{ |x|\le 2\}$, 
$\mathrm{supp}\, \theta \subset \{1/2 \le |\xi| \le 2\}$, 
and $\theta \neq 0$.
Let $N_i$
be a nonnegative integer satisfying $N_i \ge s_i$
for $i=0,1,2$.

We first prove the necessity of the condition
$s_0 \ge n/2$.  
Set
\begin{align*}
&\sigma(x,\xi_1,\xi_2)
=\varphi(x)e^{-ix\cdot(\xi_1+\xi_2)}
\langle (\xi_1,\xi_2) \rangle^{-s_0-n/2},
\\
&
\widehat{f_{1,j}}(\eta)
=\widehat{f_{2,j}}(\eta)
=2^{-jn/2} \theta (2^{-j}\eta),
\quad j \in \N_0.
\end{align*}
Since
\begin{equation}\label{sharp-smoothness-p1}
|\partial^{\alpha_0}_{x}\partial^{\alpha_1}_{\xi_1}
\partial^{\alpha_2}_{\xi_2}
\sigma(x,\xi_1,\xi_2)|
\le C_{\alpha_0, \alpha_1,\alpha_2}
\langle (\xi_1,\xi_2) \rangle^{-s_0-n/2+|\alpha_0|},
\end{equation}
in the same way as in Proof of 
Proposition \ref{classicalderivative} 
(see the argument around  
\eqref{taylorofsymbol}), 
we have  
\begin{align*}
|\Delta^{\ast}_{\boldsymbol{k}}\sigma (x,\xi_1,\xi_2)|
&\lesssim
\langle (\xi_1, \xi_2)\rangle^{-s_0-n/2+N_0}
2^{(k_0+k_1+k_2)n} 
\\
&\qquad \times
\int_{(\R^{n})^3}|\check{\psi} (2^{k_0}y)| |y|^{N_0}\,
|\check{\psi} (2^{k_1}\eta_1)|  |\eta_1|^{N_1}\,
|\check{\psi} (2^{k_2}\eta_2)|  |\eta_2|^{N_2} 
\\
&\qquad \times
\langle (\eta_1, \eta_2)\rangle^{|-s_0-n/2+N_0|}
\, dY
\\
&\approx
\langle (\xi_1, \xi_2)\rangle^{-s_0-n/2+N_0}
2^{-k_0N_0}2^{-k_1N_1}2^{-k_2N_2}  
\end{align*}
for all $k_0, k_1, k_2 \in \N$, 
where $dY=dyd\eta_1d\eta_2$. 
If we use \eqref{sharp-smoothness-p1}
with $\alpha_0=0$ and
the expression
\begin{align*}
&
\Delta^{\ast}_{\boldsymbol{k}}\sigma (x,\xi_1,\xi_2)
\\
&=
2^{(k_0+k_1+k_2)n} 
\sum_{|\alpha_1| = N_1} \frac{1}{\alpha_1!} 
\sum_{|\alpha_2| = N_2} \frac{1}{\alpha_2!} 
\\
&\qquad\times
\int_{(\R^{n})^3}
\check \psi (2^{k_0}y)\,
\check \psi (2^{k_1}\eta_1) (-\eta_1)^{\alpha_1} \,
\check \psi (2^{k_2}\eta_2) (-\eta_2)^{\alpha_2} 
\\
&\qquad\times
\int_{[0,1]^2} \bigg( \prod_{i=1}^2 
N_i (1-t_i)^{N_i-1} \bigg)
\big(\partial_{\xi_1}^{\alpha_1}
\partial_{\xi_2}^{\alpha_2} \sigma \big)
(x-y, \xi_1-t_1 \eta_1, \xi_2-t_2 \eta_2 ) 
\, dt_1dt_2 dY 
\end{align*}
instead of \eqref{taylorofsymbol}, we have
\[
|\Delta^{\ast}_{\boldsymbol{k}}\sigma (x,\xi_1,\xi_2)|
\lesssim
\langle (\xi_1, \xi_2)\rangle^{-s_0-n/2}
2^{-k_1N_1}2^{-k_2N_2} 
\] 
for $k_1, k_2 \in \N$. 
It is also easy to see that the above estimates 
actually hold for all $k_1, k_2, k_3 \in \N_0$. 
Hence, taking $0\le \theta_0 \le 1$
satisfying $s_0=N_0\theta_0$,
we have
\begin{equation}\label{sharp-smoothness-p2}
\begin{split}
|\Delta^{\ast}_{\boldsymbol{k}}\sigma (x,\xi_1,\xi_2)|
&=|\Delta^{\ast}_{\boldsymbol{k}}
\sigma (x,\xi_1,\xi_2)|^{1-\theta_0}
|\Delta^{\ast}_{\boldsymbol{k}}
\sigma (x,\xi_1,\xi_2)|^{\theta_0}
\\
&\lesssim
\left(\langle (\xi_1, \xi_2)\rangle^{-s_0-n/2}
2^{-k_1N_1}2^{-k_2N_2}\right)^{1-\theta_0}
\\
&\qquad \times
\left(\langle (\xi_1, \xi_2)\rangle^{-s_0-n/2+N_0}
2^{-k_0N_0}2^{-k_1N_1}2^{-k_2N_2}\right)^{\theta_0}
\\
&=\langle (\xi_1, \xi_2)\rangle^{-n/2}
2^{-k_0s_0}2^{-k_1N_1}2^{-k_2N_2},
\end{split}
\end{equation}
which implies that
$\sigma$ satisfies \eqref{sharp-smoothness-assump}.
Then, since
\begin{equation*}
T_{\sigma}(f_{1,j},f_{2,j})(x)
=\left(
\frac{2^{-jn}}{(2\pi)^{2n}}
\int_{(\R^n)^2}
\langle (\xi_1,\xi_2) \rangle^{-s_0-n/2}
\theta (2^{-j}\xi_1) \theta (2^{-j}\xi_2)\, d\xi_1d\xi_2\right)
\varphi(x)
\end{equation*}
and since 
\begin{equation*}
2^{-jn} 
\int_{(\R^n)^2}
\langle (\xi_1,\xi_2) \rangle^{-s_0-n/2}
\theta (2^{-j}\xi_1) \theta (2^{-j}\xi_2)\, d\xi_1d\xi_2 
\approx 
2^{j(-s_0+n/2)},
\end{equation*}
it follows from our assumption that
\[
2^{j(-s_0+n/2)}
\approx 
\|T_{\sigma}(f_{1,j},f_{2,j})\|_{L^r}
\lesssim \|f_{1,j}\|_{L^2}\|f_{2,j}\|_{L^2}
\approx 1,
\quad j \in \N_0.
\]
This is possible only if $-s_0+n/2 \le 0$,
namely $s_0 \ge n/2$.

We next prove the necessity of the condition
$s_i \ge r/n-n/2$, $i=1,2$. 
Set 
\begin{align*}
&\sigma(x,\xi_1,\xi_2)
=\langle x \rangle^{-s_1}
e^{-ix\cdot \xi_1}
\varphi(\xi_1)\varphi(\xi_2),
\\
&\widehat{f_1}(\xi_1)=\varphi(\xi_1),
\quad
\widehat{f_{2,j}}(\xi_2)
=2^{jn/2}\varphi(2^j \xi_2),
\quad j \in \N_0.
\end{align*}
Since
\[
|\partial^{\alpha_0}_{x}\partial^{\alpha_1}_{\xi_1}
\partial^{\alpha_2}_{\xi_2}
\sigma(x,\xi_1,\xi_2)|
\le C_{\alpha_0, \alpha_1,\alpha_2}
\langle x \rangle^{-s_1+|\alpha_1|}
\langle (\xi_1,\xi_2) \rangle^{-n/2},
\]
by the same argument as above,
\[
|\Delta^{\ast}_{\boldsymbol{k}}\sigma (x,\xi_1,\xi_2)|
\lesssim
\begin{cases}
\langle x \rangle^{-s_1}
\langle (\xi_1, \xi_2)\rangle^{-n/2}
2^{-k_0N_0}2^{-k_2N_2}
\\
\langle x \rangle^{-s_1+N_1}
\langle (\xi_1, \xi_2)\rangle^{-n/2}
2^{-k_0N_0}2^{-k_1N_1}2^{-k_2N_2} .
\end{cases}
\]
In the same way as in 
\eqref{sharp-smoothness-p2},
but replacing $\theta_0$ by   
$0\le \theta_1 \le 1$ 
satisfying $s_1=N_1\theta_1$, 
we have
\[
|\Delta^{\ast}_{\boldsymbol{k}}\sigma (x,\xi_1,\xi_2)|
\lesssim
\langle (\xi_1, \xi_2)\rangle^{-n/2}
2^{-k_0N_0}2^{-k_1s_1}2^{-k_2N_2},
\]
which implies that
$\sigma$ satisfies \eqref{sharp-smoothness-assump}.
On the other hand, 
since 
$\varphi(2^{j}\xi_2)\varphi(\xi_2)
=\varphi(2^{j}\xi_2)$ for $j \ge 1$,
we have 
\[
T_{\sigma}(f_1,f_{2,j})(x)
=\langle x \rangle^{-s_1}
\left(\frac{1}{(2\pi)^n}
\int_{\R^n}\varphi(\xi_1)^2\, d\xi_1\right)
2^{-jn/2}\check{\varphi}(2^{-j} x),
\quad j \ge 1,
\]
and thus 
\begin{align*}
\|T_{\sigma}(f_1,f_{2,j})\|_{L^r}
&\approx 
\|\langle x \rangle^{-s_1}
\, 2^{-jn/2}\, 
\check{\varphi}(2^{-j}x)\|_{L^r}
\ge
\|\langle x \rangle^{-s_1}
\, 2^{-jn/2}\, 
\check{\varphi}(2^{-j}x)\|_{L^r(|x| \le 2^j)}
\\
&\gtrsim 2^{j(-s_1+n/r -n/2)}. 
\end{align*}
Thus our assumption implies that
\[
2^{j(-s_1+n/r-n/2)}
\lesssim
\|T_{\sigma}(f_1,f_{2,j})\|_{L^r}
\lesssim \|f_1\|_{L^2}\|f_{2,j}\|_{L^2}
\approx 1,
\quad j \ge 1, 
\]
which is possible only if $-s_1+n/r-n/2 \le 0$,
namely $s_1 \ge n/r-n/2$.
By interchanging the roles of $\xi_1$ and $\xi_2$,
we also have $s_2 \ge n/r-n/2$.

Finally we prove the necessity of the condition
$s_1+s_2 > n/r$. 
Set
\begin{align*}
&\sigma(x,\xi_1,\xi_2)
=\langle x \rangle^{-s_1-s_2}
e^{-ix\cdot(\xi_1+\xi_2)}
\varphi(\xi_1)\varphi(\xi_2),
\\
&\widehat{f_1}
=\widehat{f_2}=\varphi . 
\end{align*}
Since
\[
|\partial^{\alpha_0}_{x}\partial^{\alpha_1}_{\xi_1}
\partial^{\alpha_2}_{\xi_2}
\sigma(x,\xi_1,\xi_2)|
\le C_{\alpha_0, \alpha_1,\alpha_2}
\langle x \rangle^{-s_1-s_2+|\alpha_1|+|\alpha_2|}
\langle (\xi_1,\xi_2) \rangle^{-n/2},
\]
by the same argument as above,
\[
|\Delta^{\ast}_{\boldsymbol{k}}
\sigma (x,\xi_1,\xi_2)|
\lesssim
\begin{cases}
\langle x \rangle^{-s_1-s_2}
\langle (\xi_1, \xi_2)\rangle^{-n/2}
2^{-k_0N_0}
\\
\langle x \rangle^{-s_1-s_2+N_1}
\langle (\xi_1, \xi_2)\rangle^{-n/2}
2^{-k_0N_0}2^{-k_1N_1}
\\
\langle x \rangle^{-s_1-s_2+N_2}
\langle (\xi_1, \xi_2)\rangle^{-n/2}
2^{-k_0N_0}2^{-k_2N_2}
\\
\langle x \rangle^{-s_1-s_2+N_1+N_2}
\langle (\xi_1, \xi_2)\rangle^{-n/2}
2^{-k_0N_0}2^{-k_1N_1}2^{-k_2N_2} .
\end{cases}
\]
Taking $0\le \theta_i\le 1$
satisfying $s_i=N_i\theta_i$ for $i=1,2$,
we have
\begin{align*}
|\Delta^{\ast}_{\boldsymbol{k}}\sigma|
&=|\Delta^{\ast}_{\boldsymbol{k}}
\sigma|^{(1-\theta_1)(1-\theta_2)}
|\Delta^{\ast}_{\boldsymbol{k}}
\sigma|^{\theta_1(1-\theta_2)}
|\Delta^{\ast}_{\boldsymbol{k}}
\sigma|^{(1-\theta_1)\theta_2}
|\Delta^{\ast}_{\boldsymbol{k}}
\sigma|^{\theta_1\theta_2}
\\
&\lesssim
\left(\langle x \rangle^{-s_1-s_2}
\langle (\xi_1, \xi_2)\rangle^{-n/2}
2^{-k_0N_0} 
\right)^{(1-\theta_1)(1-\theta_2)}
\\
&\qquad \times
\left(\langle x \rangle^{-s_1-s_2+N_1}
\langle (\xi_1, \xi_2)\rangle^{-n/2}
2^{-k_0N_0}2^{-k_1N_1}\right)^{\theta_1(1-\theta_2)}
\\
&\qquad \times
\left(\langle x \rangle^{-s_1-s_2+N_2}
\langle (\xi_1, \xi_2)\rangle^{-n/2}
2^{-k_0N_0}2^{-k_2N_2}\right)^{(1-\theta_1)\theta_2}
\\
&\qquad \times
\left(\langle x \rangle^{-s_1-s_2+N_1+N_2}
\langle (\xi_1, \xi_2)\rangle^{-n/2}
2^{-k_0N_0}2^{-k_1N_1}2^{-k_2N_2}\right)^{\theta_1\theta_2}
\\
&=\langle (\xi_1, \xi_2)\rangle^{-n/2}
2^{-k_0N_0}2^{-k_1s_1}2^{-k_2s_2},
\end{align*}
which implies that
$\sigma$ satisfies \eqref{sharp-smoothness-assump}.
Therefore,
since
\[
T_{\sigma}(f_1,f_2)(x)
=\langle x \rangle^{-s_1-s_2}
\prod_{i=1}^2\left(\frac{1}{(2\pi)^n}
\int_{\R^n}\varphi(\xi_i)^2\, d\xi_i\right),
\]
it follows from our assumption that
$\langle x \rangle^{-s_1-s_2}$
belongs to $L^r$.
This is possible only if $r(-s_1-s_2) < -n$,
namely $s_1+s_2 >n/r$.
\end{proof}

In the corollary below,   
$BS^{-n/2}_{0,0} (s_0, s_1, s_2; \R^n)$ 
denotes the class 
$BS ^{W, \ast}_{0,0} (s_0, s_1, s_2; \R^n)$ of 
Definition \ref{defBSW2} 
for 
$W(\xi_1, \xi_2)=W_{-n/2}(\xi_1, \xi_2) 
=\langle (\xi_1, \xi_2) \rangle^{-n/2} $.

\begin{cor}\label{sharp-smoothness-2}
Let $1 \le r \le 2$
and $\boldsymbol{s}=(s_0,s_1,s_2)
\in [0, \infty)^3$. 
Assume all bilinear pseudo-differential operators 
$T_{\sigma}$ with 
$\sigma \in 
BS^{-n/2}_{0,0}(s_0, s_1, s_2; \R^n)$ are 
bounded from $L^2 \times L^2$
to $L^r$. 
Then 
$s_0 \ge n/2$,
$s_1,s_2 \ge n/r-n/2$, 
and $s_1+s_2 \ge n/r$.
\end{cor}

\begin{proof}
Observe that 
all $\sigma$ satisfying 
\eqref{sharp-smoothness-assump} 
with $s_i$ replaced by 
$s_i+ \epsilon$ with $\epsilon>0$ 
belong to 
$BS ^{ -n/2 }_{0,0} (s_0, s_1, s_2; \R^n)$.  
Hence, 
if the assumption of the corollary holds, 
then, by Proposition \ref{sharp-smoothness}, 
we must have 
$s_0 + \epsilon \ge n/2$,
$s_1 + \epsilon ,s_2 + \epsilon  \ge n/r-n/2$, 
and $s_1 + \epsilon +s_2 + \epsilon  > n/r$ for $\epsilon >0$. 
Since $\epsilon >0$ is arbitrary, we obtain 
the conclusion. 
\end{proof}

\begin{acknowledgement}
The authors are grateful to Neal Bez for 
valuable discussions concerning the proof of Proposition \ref{weightL4weak}. 
\end{acknowledgement}


\end{document}